\documentclass[11pt]{article}
\usepackage{epic,latexsym,amssymb}
\usepackage{amsfonts}
\usepackage{amscd}
\usepackage{amsmath}
\usepackage{graphicx}
\usepackage{color}
\usepackage{caption,
caption}
\usepackage{tikz}
\usepackage{epigraph}
\usepackage{amsthm}
\usepackage{bbm}
\usepackage{amsfonts}

\usepackage[margin=0.25in]{geometry}
\usepackage{pgfplots}
\pgfplotsset{width=10cm,compat=1.9}

\usepackage{epic,latexsym,amssymb}
\usepackage{tikz}
\usepackage[bold,full]{complexity}
\usepackage{tikz,epsf}
\usepackage{hyperref}
\usepackage{graphicx}[width=\textwidth]
\usepackage{amsmath}
\usepackage{listings}

\usepackage{amsfonts}
\usepackage{amscd}
\usepackage{amsmath}
\usepackage{graphicx}
\usepackage{color}
\usepackage{caption,subcaption}
\usepackage{lineno}
\usepackage[ruled,vlined]{algorithm2e}

\usepackage{lineno}

\newtheorem{thm}{Theorem}

\newtheorem{ob}[thm]{Observation}
\newtheorem{prop}[thm]{Proposition}
\newtheorem{lem}[thm]{Lemma}

\newtheorem{cor}[thm]{Corollary}

\theoremstyle{definition}
\newtheorem{definition}{Definition}
\newtheorem{example}{Example}

\newcommand{\cprod}{\,\Box\,}
\newcommand{\comp}{\mathrm{comp}}

\newcommand{\QEDmark}{\mbox{\textsc{qed}}}
\newcommand{\proofStarter}[1]{\textsc{#1} }

\def\vertex(#1){\put(#1){\circle*{2}}}
\def\vertexo(#1){\put(#1){\circle{2}}}
\def\vert(#1){\put(#1){\circle*{1.5}}}
\def\verto(#1){\put(#1){\circle{1.5}}}
\def\lab(#1)#2{\put(#1){\makebox(0,0)[c]{#2}}}
\definecolor{DarkGreen}{rgb}{0.2, 0.6, 0.3}

\definecolor{electricindigo}{rgb}{0.44, 0.0, 1.0}

\usepackage{soul}

\let\oldenumerate\enumerate
\renewcommand{\enumerate}{
  \oldenumerate
  \setlength{\itemsep}{0.5pt}
  \setlength{\parskip}{0pt}
  \setlength{\parsep}{0pt}
}

\title{Connected forcing density and related problems}
\author
    {Boris Brimkov \\
    \small Department of Mathematics,  \\
    \small Statistics, and Physics \\
    \small Slippery Rock University \\
    \small \texttt{boris.brimkov@sru.edu} \\ 
    \and
    Randy Davila\thanks{Corresponding author.} \\
    \small Department of Computational Applied \\ 
    \small Mathematics \& Operations Research \\ 
    \small Rice University \\ 
    \small \texttt{rrd6@rice.edu} 
    \and
    Houston Schuerger \\
    \small Department of Mathematics \\ 
    \small University of Texas-Permian Basin \\ 
    \small \texttt{schuerger\_h@utpb.edu} 
    }

\date{}

\begin{document}
\maketitle

\begin{abstract}
A connected forcing set of a graph is a zero forcing set that induces a connected subgraph. In this paper, we introduce and study CF-dense graphs --- graphs in which every vertex belongs to some minimum connected forcing set. We identify several CF-dense graph families and investigate the relationships between CF-density and analogous notions in zero forcing and total forcing. We also characterize CF-dense trees and give a formula for the number of distinct connected forcing sets in trees. Finally, we analyze when CF-density is preserved under graph operations such as Cartesian products, joins, and coronas.
\end{abstract}

\section{Introduction}
\label{sec:intro}

Let $G$ be a simple, undirected, and finite graph, with vertex set $V(G)$, where each vertex $v \in V(G)$ is colored either blue or white. The \emph{zero forcing color change rule} is a dynamic process in which a blue vertex with exactly one white neighbor \emph{forces} that neighbor to become blue. A set $S \subseteq V(G)$ is called a \emph{zero forcing set} if, starting from $S$ as the set of initially blue vertices, repeated application of the color change rule results in all vertices of $G$ becoming blue. The \emph{zero forcing number} of $G$, denoted $Z(G)$, is the minimum cardinality of such a set. This parameter was introduced in \cite{AIM-Workshop} as a combinatorial tool to bound the maximum nullity of a graph.

Variants and analogues of zero forcing have appeared independently across a range of disciplines, including PMU placement in power networks \cite{BruneiHeath,powerdom3}, target set selection in social influence models \cite{target1,target3,target2}, quantum control theory \cite{quantum1}, and fast-mixed searching \cite{fast_mixed_search}. Beyond these applications, zero forcing has been used in coding theory, logic circuit design, and the modeling of contagion processes in epidemiology and social networks; see \cite{zf_tw,quantum1,logic1,zf_np} and the monograph \cite{HoLiSh-zero-forcing-book}. As a graph invariant, $Z(G)$ has also attracted considerable attention for its connections to classical parameters such as \emph{independence number}, \emph{vertex cover number}, \emph{domination number}, \emph{chromatic number}, and Hamiltonicity, and for its behavior in structured graph families such as claw-free graphs; see, e.g.,~\cite{k-forcing, upper-total-dom, vc, Da-15, Da-19, DaHe-20b, DaHe-20a, zfclaw, indep, join} and the monograph~\cite{HaHeHe2024}.


Several variants of zero forcing have been introduced to incorporate additional structural or application-driven constraints. For instance, a zero forcing set that induces a subgraph without isolated vertices is called a \emph{total forcing set}, and the minimum size of such a set in a graph $G$ is the \emph{total forcing number}, denoted $Z_t(G)$. Similarly, a zero forcing set that induces a connected subgraph is called a \emph{connected forcing set}, and its minimum size is the \emph{connected forcing number}, denoted $Z_c(G)$. These variants arise naturally in applications such as optimizing PMU placement in power networks, where minimizing wiring complexity is a key concern. Structural and algorithmic properties of total and connected forcing have been extensively studied; see \cite{BrDa, cf-complexity, Da-19, DaHe19c, DaHe19a, DaHe19b, Da-connected-forcing, new-connected-zero-forcing}. Notably, computing $Z(G)$, $Z_t(G)$, and $Z_c(G)$ is generally intractable, as the associated decision problems are $\mathcal{NP}$-complete~\cite{NP-Complete, Da-19, DaHe19b}.


Beyond determining the minimum size of a zero forcing, total forcing, or connected forcing set, it is also of interest to study graphs in which \emph{every} vertex belongs to some such minimum set. This idea was first formalized in~\cite{DaHePe2023a}, where the notions of \emph{ZF-dense} and \emph{TF-dense} graphs were introduced. A graph is ZF-dense (respectively, TF-dense) if each vertex belongs to some minimum zero forcing set (respectively, total forcing set). A graph that is both ZF-dense and TF-dense is said to be \emph{ZTF-dense}.

In this paper, we introduce and investigate the analogous concept of \emph{CF-dense} graphs~-- graphs in which every vertex belongs to some minimum connected forcing set. While the motivations behind CF-density align with those of its analogues, our results reveal that CF-dense graphs can differ markedly in both structure and parameter values from ZF-dense and TF-dense graphs. To capture the overlap among all three notions, we define a graph to be \emph{ZTCF-dense} if it is simultaneously ZF-dense, TF-dense, and CF-dense. Intermediate notions such as \emph{ZCF-dense} and \emph{TCF-dense} are defined analogously. We address the distinct behavior of CF-density by characterizing CF-dense trees and identifying several new families of ZTCF-dense graphs.

We also investigate how the connected and total forcing numbers behave under common graph operations, including Cartesian products, joins, and coronas. These results yield new families of CF-dense and ZTCF-dense graphs and lead to several new formulas for computing connected and total forcing numbers in product constructions. In particular, we show that under certain structural conditions, both CF-density and ZTCF-density are preserved under these operations.

The remainder of the paper is organized as follows. In Section~\ref{sec:preliminaries} we give notation and terminology used throughout this manuscript. In Section~\ref{sec:uniqueness}, we examine graphs with unique connected forcing sets and characterize the number of very large or very small connected forcing sets a graph can have. Section~\ref{sec:basics} identifies several new families of CF-dense graphs. In Section~\ref{sec:trees-and-unicyclic}, we characterize all CF-dense trees and provide a complete enumeration of their minimum connected forcing sets. Section~\ref{sec:cartesian-alt} explores how various graph operations affect CF-density and ZTCF-density. We conclude with final remarks and open questions in Section~\ref{sec:conclusion}.

\section{Preliminaries} 
\label{sec:preliminaries}

Throughout this paper, all graphs are assumed to be simple, undirected, and finite. Let $G$ be a graph with vertex set $V(G)$ and edge set $E(G)$. The \emph{order} of $G$ is $n(G) = |V(G)|$, and its \emph{size} is $m(G) = |E(G)|$. A graph is said to be \emph{nontrivial} if $n(G) \geq 2$. Two vertices $v, w \in V(G)$ are \emph{adjacent} (or \emph{neighbors}) if $vw \in E(G)$. The \emph{open neighborhood} of a vertex $v$ is denoted by $N_G(v)$ and consists of all vertices adjacent to $v$.
The \emph{degree} of a vertex $v$ is $d_G(v) = |N_G(v)|$. The minimum degree among the vertices in $G$ is denoted $\delta(G)$.  Occasionally, for notations such as $d_G(v)$ or $N_G(v)$, the $G$ is dropped if it is clear from context that the graph in question is $G$.

For a set of vertices $S \subseteq V(G)$, the \emph{open neighborhood} of $S$ is defined as $N_G(S) = \bigcup_{v \in S} N_G(v)$.
The \emph{boundary} of $S$, denoted $\partial_G(S)$, is the set of vertices in $N_G(S)$ that do not belong to $S$; that is, $\partial_G(S) = N_G(S) \setminus S$. A graph $G$ is said to be \emph{connected} if there exists a path between every pair of vertices. The maximal connected subgraphs of $G$ are called its \emph{components} and $\comp(G)$ denotes the number of components $G$ has. A graph $H$ is a \emph{subgraph} of graph $G$, denoted $H \subseteq G$, if $V(H) \subseteq V(G)$ and $E(H) \subseteq E(G)$. For a subset $S \subseteq V(G)$, the subgraph \emph{induced} by $S$ is denoted $G[S]$. A vertex $v \in V(G)$ is a \emph{cut-vertex} if $G[V(G) \setminus \{v\}]$ is disconnected.  Two graphs $G$ and $H$ are said to be isomorphic, denoted $G \simeq H$, provided there exists a bijective function $f:V(G) \rightarrow V(H)$ such that given $u,v \in V(G)$, $uv \in E(G)$ if and only if $f(u)f(v) \in E(H)$.

The complete graph, path, and cycle on $n$ vertices are denoted by $K_n$, $P_n$, and $C_n$, respectively. A \emph{tree} is a connected graph with no cycles, and a \emph{forest} is any acyclic graph. In a tree, a vertex of degree one is called a \emph{leaf}, and a vertex adjacent to a leaf is a \emph{support vertex}. A support vertex adjacent to at least two leaves is referred to as a \emph{strong support vertex}. The \emph{diamond graph} is the graph obtained by removing one edge from $K_4$. A \emph{trivial graph} is a graph on exactly one vertex and an \emph{isolate-free} graph is a graph with no trivial graph as a component. 

For any graph-theoretic terminology not defined here, we refer the reader to~\cite{HaHeHe2024}. For zero forcing conventions, we refer to the monograph~\cite{HoLiSh-zero-forcing-book}. We will also use the standard notation $[k] = \{1, \dots, k\}$.

\section{Uniqueness of Connected Forcing Sets}
\label{sec:uniqueness}

A central theme of this paper is the study of CF-dense graphs -- graphs in which every vertex belongs to some minimum connected forcing set. As a natural complement to this, we begin by exploring the uniqueness and extremal behavior of connected forcing sets. In particular, we show that if a graph $G$ of order $n$ has a unique connected forcing set of a given size, then that size must be $n$ or $Z_c(G)$. We also characterize graphs based on the number of very large or very small connected forcing sets they admit, shedding light on the structural rigidity or flexibility of their forcing behavior.

To formalize our study of connected forcing sets of various sizes, we introduce the following notation.

\begin{definition}
\label{def_zc}
Let $G$ be a connected graph of order $n$. For each integer $1 \leq i \leq n$, let $z_c(G; i)$ denote the number of distinct connected forcing sets of $G$ of size $i$.
\end{definition}

In particular, $z_c(G; Z_c(G))$ counts the number of minimum connected forcing sets of $G$. The following proposition gives a natural range for this quantity.

\begin{prop}
\label{prop:cf_prop_bound}
For any connected graph $G$, 
\[
1 \leq z_c(G; Z_c(G)) \leq \binom{n}{Z_c(G)},
\]
and both bounds are tight.
\end{prop}
\begin{proof}
The lower bound follows from the fact that $G$ must have at least one connected forcing set of minimum size by the definition of $Z_c(G)$. The upper bound follows from the fact that there are at most $\binom{n}{Z_c(G)}$ subsets of $V(G)$ of size $Z_c(G)$. 

To show both bounds are tight, note that the graph shown in Figure~\ref{fig:extensions} has a unique minimum connected forcing set, so $z_c(G; Z_c(G)) = 1$. On the other hand, for the complete graph $K_n$, we have $Z_c(K_n) = n - 1$, and every $(n-1)$-subset of vertices forms a connected forcing set. Thus, $z_c(K_n; Z_c(K_n)) = \binom{n}{n-1} = \binom{n}{Z_c(K_n)}$.
\end{proof}



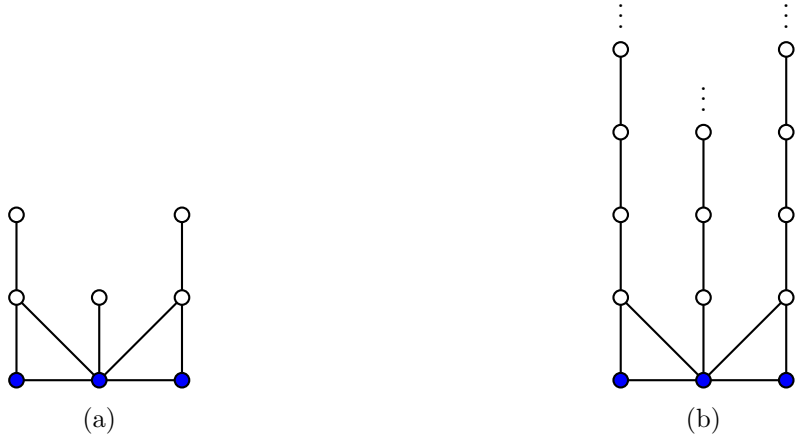
\begin{figure}[htb]
\centering

\begin{subfigure}[b]{0.45\textwidth}
\centering
\begin{tikzpicture}[scale=1.1,style=thick,x=1cm,y=1cm]
\def\vr{2.5pt}

\path (0, 1) coordinate (v1);
\path (1, 1) coordinate (v2);
\path (2, 1) coordinate (v3);
\path (0, 2) coordinate (v4);
\path (1, 2) coordinate (v5);
\path (2, 2) coordinate (v6);
\path (0, 3) coordinate (v7);
\path (2, 3) coordinate (v8);

\draw (v1) -- (v2) -- (v3);
\draw (v1) -- (v4);
\draw (v2) -- (v4) -- (v7);
\draw (v2) -- (v5);
\draw (v2) -- (v6) -- (v3);
\draw (v6) -- (v8);

\foreach \i in {v1,v2,v3} \draw (\i) [fill=blue] circle (\vr);
\foreach \i in {v4,v5,v6,v7,v8} \draw (\i) [fill=white] circle (\vr);
\end{tikzpicture}
\caption*{(a)}
\end{subfigure}
\hfill
\begin{subfigure}[b]{0.45\textwidth}
\centering
\begin{tikzpicture}[scale=1.1,style=thick,x=1cm,y=1cm]
\def\vr{2.5pt}

\path (0, 1) coordinate (v1);
\path (1, 1) coordinate (v2);
\path (2, 1) coordinate (v3);
\path (0, 2) coordinate (v4);
\path (1, 2) coordinate (v5);
\path (2, 2) coordinate (v6);

\path (0, 3) coordinate (v7);
\path (0, 4) coordinate (v71);
\path (0, 5) coordinate (v72);

\path (1, 3) coordinate (v51);
\path (1, 4) coordinate (v52);

\path (2, 3) coordinate (v8);
\path (2, 4) coordinate (v81);
\path (2, 5) coordinate (v82);

\draw (v1) -- (v2) -- (v3);
\draw (v1) -- (v4);
\draw (v2) -- (v4) -- (v7) -- (v71) -- (v72);
\draw (v2) -- (v5) -- (v51) -- (v52);
\draw (v2) -- (v6) -- (v3);
\draw (v6) -- (v8) -- (v81) -- (v82);

\node at (0, 5.5) {\small$\vdots$};
\node at (1, 4.5) {\small$\vdots$};
\node at (2, 5.5) {\small$\vdots$};

\foreach \i in {v1,v2,v3}
  \draw (\i) [fill=blue] circle (\vr);

\foreach \i in {v4,v5,v6,v7,v8,v71,v81,v72,v82,v51,v52}
  \draw (\i) [fill=white] circle (\vr);
\end{tikzpicture}
\caption*{(b)}
\end{subfigure}

\caption{(a) A graph with a unique minimum connected forcing set (shown in blue). (b) An extension of this graph into an infinite family, each with a unique minimum connected forcing set.}
\label{fig:extensions}
\end{figure}

The graph in Figure \ref{fig:extensions} is also an example of a graph with a unique minimum total forcing set. On the other hand, a (non-trivial) graph cannot have a unique minimum zero forcing set, since reversing the forcing chains associated with a zero forcing set produces another zero forcing set. We now strengthen Proposition~\ref{prop:cf_prop_bound} by showing that only a minimum and maximum connected forcing set can be unique.

\begin{lem}
\label{prop_path_cf_poly}
For $n\geq 2$, $z_c(P_n;i)=
\begin{cases}
2 &\text{if\; } i=1\\
n-i+1 &\text{if\; } 2\leq i\leq n.
\end{cases}$
\end{lem}
\begin{proof}
Either endpoint of $P_n$ is a connected forcing set of size 1, so $z_c(P_n;1)=2$. Label the vertices of $P_n$ in order starting from one endpoint and let $R$ be a connected forcing set of size $i$. There are $n-(i-1)$ ways to choose the vertex in $R$ with the smallest label $j$. This choice uniquely determines $R$, since the other $i-1$ vertices in $R$ must be the vertices with labels $j+1,\ldots,j+(i-1)$. Thus, $z_c(P_n;i)=n-i+1$ for $2\leq i\leq n$.
\end{proof}

\begin{thm}
$z_c(G;i)=1$ only if $i=Z_c(G)$ or $i=n$. 
\end{thm}
\begin{proof}
Suppose $G$ has a unique connected forcing set $R$ of size $i$, where $Z_c(G)<i<n$. Let $R'$ be a connected forcing set of size $i-1$. Then $R'\subset R$, since otherwise $R'$ together with any neighbor of $R'$ forms a connected forcing set of size $i$ different from $R$. Let $\{v\}=R\backslash R'$. Then $\partial_G(R')=\{v\}$, since otherwise $R'$ together with either of its two or more neighbors would form a connected forcing set of size $i$. Moreover, for all $u \in \partial_G(R)$, $N_G(u)\cap R=\{v\}$, since if some $u \in \partial_G(R)$ was adjacent to a vertex in $R$ other than $v$, $v$ would not be the only vertex in $\partial_G(R')$. 

If $R'$ consists of a single vertex, then $G$ is a path, and by Proposition \ref{prop_path_cf_poly}, there is no unique connected forcing set of size $i<n$; thus, $R'$ has at least two vertices. Let $w$ be a non-cut vertex of $G[R]$ different from $v$, and let $u$ be a vertex in $\partial_G(R)$. We claim $R\backslash\{w\}\cup \{u\}$ is a connected forcing set of $G$. This set is connected by construction, and it is forcing since $w$ can be forced in the first timestep by any of its neighbors in $R'$. This contradicts $R$ being the unique connected forcing set of size $i$.
\end{proof}

\noindent We conclude this section by characterizing $z_c(G;i)$ for extremal values of $i$.

\begin{prop}
\label{thm_zf_polynomial_properties}
If $G$ is a connected graph of order $n$, then
\begin{enumerate}
\item[\textup{(1)}] $z_c(G; n) = 1$,
\item[\textup{(2)}] $z_c(G; n-1) = \left|\left\{ v \in V : \comp(G \setminus \{v\}) = 1\right\}\right|$,
\item[\textup{(3)}] 
$
z_c(G; n-2) = \left| \left\{
\begin{aligned}
\{u, v\} \subset V :\ & u \ne v,\ \comp(G \setminus \{u, v\}) = 1, \\
& N_G(u) \setminus \{v\} \ne N_G(v) \setminus \{u\}
\end{aligned}
\right\} \right|,
$
\item[\textup{(4)}] 
$
z_c(G; 1) = 
\begin{cases}
2 & \text{if } G \simeq P_n \text{ for } n \ge 2, \\
1 & \text{if } G \simeq P_1, \\
0 & \text{otherwise}.
\end{cases}
$
\end{enumerate}
\end{prop}

\proof The numbers of the proofs below correspond to the numbers in the statement of the theorem. 
\begin{enumerate}
\item[(1)] The only connected forcing size of size $n$ is $V$.
\item[(2)] Since $G$ is connected, any non-cut vertex $v$ has a neighbor that can force $v$. Thus, each set which excludes one non-cut vertex of $G$ is a connected forcing set of size $n-1$; moreover, no set which excludes a cut vertex is a connected forcing set. 
\item[(3)] Let $u,v$ be two non-isolated vertices of $G$; if $N_G(u)\backslash\{v\}\neq N_G(v)\backslash\{u\}$, there is a vertex $w$ adjacent to one of $u$ and $v$, but not the other. Suppose $uw \in E(G)$; then, $w$ can force $u$ and any neighbor of $v$ can force $v$ (since $d(v)\neq 0$). Thus, any pair of vertices $u,v$ satisfying these conditions can be excluded from a zero forcing set of size $n-2$. On the other hand, a pair of vertices $u,v$ that does not satisfy these conditions cannot be excluded from a zero forcing set, since every vertex that is adjacent to one will be adjacent to the other, and hence no vertex will be able to force $u$ or $v$. Moreover, any pair of vertices that form a separating set cannot be excluded from any connected forcing set of size $n-2$.  
\item[(4)] The only graph with connected forcing number 1 is $P_n$; thus if $G\not\simeq P_n$, $z_c(G;1)=0$. If $G\simeq P_n$ and $n\geq 2$, either end of the path is a connected forcing set; if $n=1$, there is a single connected forcing set.
\end{enumerate}

We note that since graphs $G$ with $Z_c(G)=2$ have been completely characterized in \cite{extremalcf}, it is possible to obtain a closed form expression for $z_c(G;2)$. However, such an expression would have numerous cases  depending on the structure of $G$, and we have left this description for future work.

\section{Families of ZTCF-Dense Graphs}\label{sec:basics}

In this section, we identify several families of graphs that are ZTCF-dense. Some of the arguments are extensions of results in~\cite{DaHePe2023a} on ZF-dense and TF-dense graphs.

\begin{ob}
Every cycle graph $C_n$ is ZTCF-dense.
\end{ob}
\begin{proof}
Any two adjacent vertices of $C_n$ form a minimum zero forcing, total forcing, and connected forcing set.
\end{proof}

\begin{ob}\label{ob:kn}
The complete graph $K_n$ is ZCF-dense for $n \geq 1$, and ZTCF-dense for $n\geq 2$.
\end{ob}
\begin{proof}
If $n\geq 3$, any $n-1$ vertices of $K_n$ form a minimum zero forcing, total forcing, and connected forcing set. $K_2$ is clearly ZTCF-dense, and $K_1$ is ZCF-dense but not TF-dense since $K_1$ has no total forcing set.
\end{proof}

The following chain of inequalities relating the minimum degree and the zero forcing, total forcing, and connected forcing numbers is well known and follows from the definitions.

\begin{ob}\label{ob:min_degree}
If $G$ is a connected graph different from a path with minimum degree $\delta(G)$, then
\[\delta(G)\leq Z(G) \leq Z_t(G) \leq Z_c(G).\]
\end{ob}

A \emph{wheel graph} $W_n$ ($n\geq 4$) is the graph formed by connecting a single central vertex $v_c$ to all vertices of a cycle $C_{n-1}$. 

\begin{figure}[htb]
\centering

\begin{subfigure}[b]{0.45\textwidth}
\centering
\begin{tikzpicture}[scale=1.1,style=thick,x=1cm,y=1cm]
\def\vr{2.5pt}
\path (0:1) coordinate (v1);
\path (60:1) coordinate (v2);
\path (120:1) coordinate (v3);
\path (180:1) coordinate (v4);
\path (240:1) coordinate (v5);
\path (300:1) coordinate (v6);
\path (0,0) coordinate (vc);
\draw (v1) -- (v2) -- (v3) -- (v4) -- (v5) -- (v6) -- (v1);
\foreach \i in {1,2,3,4,5,6}
  \draw (vc) -- (v\i);
\foreach \i in {3,4,5,6}
  \draw (v\i) [fill=white] circle (\vr);
\foreach \i in {1,2}
  \draw (v\i) [fill=blue] circle (\vr);
\draw (vc) [fill=blue] circle (\vr);
\end{tikzpicture}
\caption{The wheel graph $W_6$}
\end{subfigure}
\hfill
\begin{subfigure}[b]{0.45\textwidth}
\centering
\begin{tikzpicture}[scale=1.1,style=thick,x=1cm,y=1cm]
\def\vr{2.5pt}
\path (-1,0) coordinate (v0000);
\path (0,0) coordinate (v1000);
\path (-1,1) coordinate (v0100);
\path (0,1) coordinate (v1100);
\path (-0.5,0.5) coordinate (v0010);
\path (0.5,0.5) coordinate (v1010);
\path (-0.5,1.5) coordinate (v0110);
\path (0.5,1.5) coordinate (v1110);
\path (1.5,-0.5) coordinate (v0001);
\path (2.5,-0.5) coordinate (v1001);
\path (1.5,0.5) coordinate (v0101);
\path (2.5,0.5) coordinate (v1101);
\path (2,0) coordinate (v0011);
\path (3,0) coordinate (v1011);
\path (2,1) coordinate (v0111);
\path (3,1) coordinate (v1111);
\foreach \i/\j in {
0000/1000,0000/0100,0100/1100,1000/1100,
0000/0010,0100/0110,1100/1110,1000/1010,
0010/1010,0010/0110,0110/1110,1010/1110,
0001/1001,0001/0101,0101/1101,1001/1101,
0001/0011,0101/0111,1101/1111,1001/1011,
0011/1011,0011/0111,0111/1111,1011/1111,
0000/0001,1000/1001,0100/0101,1100/1101,
0010/0011,1010/1011,0110/0111,1110/1111}
\draw (v\i) -- (v\j);
\foreach \i in {0000,1000,0100,1100,0010,1010,0110,1110}
  \draw (v\i) [fill=blue] circle (\vr);
\foreach \i in {0001,1001,0101,1101,0011,1011,0111,1111}
  \draw (v\i) [fill=white] circle (\vr);
\end{tikzpicture}
\caption{The hypercube $Q_4$}
\end{subfigure}

\vspace{1em} 

\begin{subfigure}[b]{0.6\textwidth}
\captionsetup{skip=0pt}
\centering
\begin{tikzpicture}[scale=.6,style=thick,x=1cm,y=1cm]
\def\vr{3.5pt}
\path (-.5, 0.0) coordinate (d11);
\path (0.5, 1.0) coordinate (d12);
\path (1.5, 0.0) coordinate (d13);
\path (0.5, -1.0) coordinate (d14);
\path (2.5, 0.0) coordinate (d21);
\path (3.5, 1.0) coordinate (d22);
\path (4.5, 0.0) coordinate (d23);
\path (3.5, -1.0) coordinate (d24);
\path (5.5, 0.0) coordinate (d31);
\path (6.5, 1.0) coordinate (d32);
\path (7.5, 0.0) coordinate (d33);
\path (6.5, -1.0) coordinate (d34);
\path (8.5, 0.0) coordinate (d41);
\path (9.5, 1.0) coordinate (d42);
\path (10.5,0.0) coordinate (d43);
\path (9.5, -1.0) coordinate (d44);
\foreach \i/\j in {
d11/d12,d12/d13,d11/d14,d14/d13,d12/d14,d13/d21,
d21/d22,d22/d23,d21/d24,d24/d23,d22/d24,d23/d31,
d31/d32,d32/d33,d31/d34,d34/d33,d32/d34,d33/d41,
d41/d42,d42/d43,d41/d44,d44/d43,d42/d44}
\draw (\i) -- (\j);
\draw (d11) to[out=-125,in=-65] (d43);
\foreach \i in {d13,d14,d21,d22,d23,d31,d33,d34,d41,d44}
  \draw (\i) [fill=blue] circle (\vr);
\foreach \i in {d11,d12,d24,d32,d42,d43}
  \draw (\i) [fill=white] circle (\vr);
\end{tikzpicture}
\caption{The diamond necklace $N_4$}
\end{subfigure}

\caption{Three graphs with minimum connected forcing sets shown in blue: (a) the wheel graph $W_6$, (b) the hypercube $Q_4$, and (c) the diamond necklace $N_4$.}
\label{fig:forcing-abc-layout}
\end{figure}
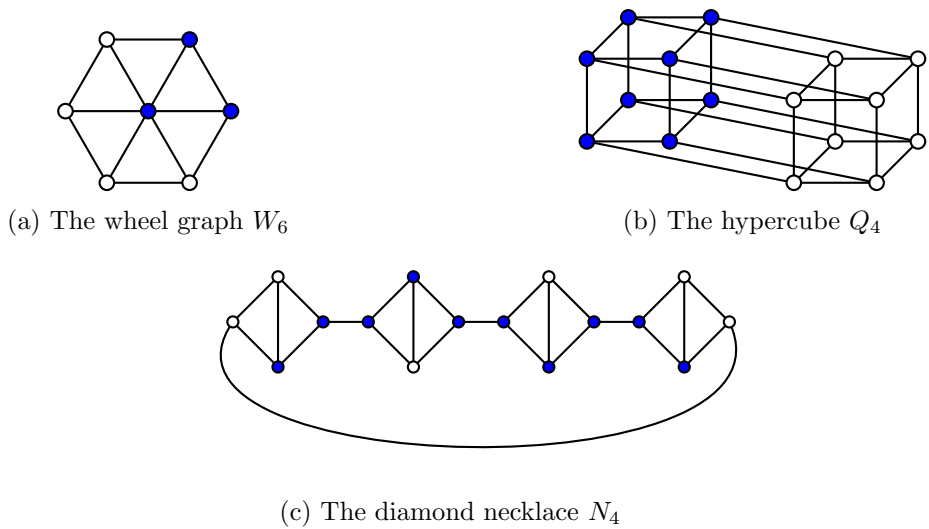

\begin{prop}\label{prop:wheels}
If $G$ is a wheel graph, then $G$ is ZTCF-dense. 
\end{prop}

\begin{proof}
Let $S$ be a set consisting of the central vertex $v_c$ of $G$ and two other adjacent vertices of $G$ (see Figure~\ref{fig:forcing-abc-layout} (a)). $G[S]$ is connected, and is also a zero forcing set since the two adjacent vertices on the $C_{n-1}$ subgraph of $G$ can force all other vertices in that cycle. By Observation~\ref{ob:min_degree}, $3=\delta(G)\leq Z_c(G)\leq |S|$, so $Z_c(G)=3$. Moreover, since $S$ can be chosen to contain any vertex of $G$, $G$ is CF-dense. 
In \cite{DaHePe2023a} it is shown that $G$ is ZTF-dense; therefore, $G$ is ZTCF-dense.
\end{proof}

The $k$-dimensional \emph{hypercube} $ Q_k $ is defined recursively as follows:
$Q_1 = K_2$, and for $ k \ge 2 $, $ Q_k = Q_{k-1} \cprod K_2 $, where $ \cprod $ denotes the Cartesian product of graphs.

\begin{prop}\label{prop:cubes}
If $G$ is a hypercube, then $G$ is ZTCF-dense. 
\end{prop}

\begin{proof}
Let $ G = Q_k = Q_{k-1} \cprod K_2 $. Let $S$ be the set consisting of all vertices in one of the two copies of $ Q_{k-1}$ (see Figure~\ref{fig:forcing-abc-layout} (b)). $G[S]$ is connected, and is also a zero forcing set since each vertex can force its neighbor in the other copy of $Q_{k-1}$ in the first timestep. In~\cite{Pe12} it was shown that $ Z(Q_k) = 2^{k-1} $. Thus, by Observation \ref{ob:min_degree}, $2^{k-1}=Z(G)\leq Z_c(G)\leq |S|= 2^{k-1} $, so $Z_c(G)= 2^{k-1}$. Moreover, since $S$ can be chosen to be either of the two copies of $Q_{k-1}$, every vertex of $G$ belongs to some minimum connected  forcing set. Thus, $ G $ is CF-dense. In \cite{DaHePe2023a} it is shown that hypercubes are ZTF-dense; therefore, $G$ is  ZTCF-dense.
\end{proof}

A \emph{star} $S_n$ is the complete bipartite graph $K_{1,n-1}$. The degree $n-1$ vertex of $S_n$ is called the central vertex.

\begin{prop}\label{prop:star}
The star graph $S_n$ is TCF-dense for $n\geq 4$ (but not ZF-dense). 
\end{prop}
\begin{proof}
Let $v$ be any leaf of $S_n$. Any set $S$ consisting of all vertices of $S_n$ except $v$ is connected and forcing, and is also minimum since when $n\geq 4$, omitting two leaves will cause the set to not be forcing, and omitting the central vertex  will cause the set to have isolates. Thus, $S_n$ is TCF-dense. On the other hand, the central vertex of $S_n$ is not contained in any minimum zero forcing set. Thus, $S_n$ is not ZF-dense. Note also that $S_3$ is not CF-dense, since then the central vertex would not be contained in any minimum connected forcing set.
\end{proof}

A \emph{complete multipartite graph} is a graph whose vertex set can be partitioned into $ k\geq 2 $ disjoint independent sets $ V_1, \ldots, V_k $, and every vertex in $ V_i $ is adjacent to every vertex in $ V_j $ for all $ i \neq j $. 

\begin{prop}\label{prop:multipartite}
If $ G $ is a complete multipartite graph that is not a star, then $ G $ is ZTCF-dense. 
\end{prop}

\begin{proof}
Let $ V_1,\ldots, V_k $ be the independent sets that partition the vertices of $G$, with $ 1 \le |V_1| \le \cdots \le |V_k| $. If $ |V_k| = 1 $, then $ G = K_n $ ($n\geq 2$), which is ZTCF-dense by Observation~\ref{ob:kn}. Now suppose that $ |V_k| \ge 2 $. Since $ G $ is not a star,  either $ k \ge 3 $ or $ k = 2 $ and $ |V_1| \ge 2 $. 

Suppose first that $ k \ge 3 $. Let $u,v\in V_k$, $a\in V_i$, and $b\in V_j$ (where $i$, $j$, and $k$ are all different). Let $S=V(G)\setminus \{u,a\}$. Then, $G[S]$ is connected since $x,v,y$ is a path between any two vertices $x,y$ not in $V_k$, and $x,b,y$ is a path between any two vertices $x,y$ in $V_k$. 

Now suppose that $ k = 2 $ and $ |V_1| \ge 2 $. Let $u,v\in V_1$ and $a,b\in V_2$. Let $S=V(G)\setminus \{u,a\}$. Then, $G[S]$ is connected since $x,v,y$ is a path between any two vertices $x,y$ in $V_2$, and $x,b,y$ is a path between any two vertices $x,y$ in $V_1$.

In either case, $S$ is a zero forcing set since first $v$ can force $a$, and then $a$ can force $u$. Thus, $ S $ is a connected forcing set. In \cite{AIM-Workshop} it was shown that $Z(G) = n - 2$, and since $ G $ is not a path, by Observation \ref{ob:min_degree}, $ Z(G) \le Z_c(G)$. Thus, $ n-2\leq Z_c(G)  \le |S|=n - 2$, so $ Z_c(G) = n - 2 $. Moreover, in either case, $S$ can be chosen to contain any vertex of $G$. Thus, every vertex of $G$ belongs to at least one minimum connected forcing set, so $ G $ is CF-dense. Finally, it was shown in~\cite{DaHePe2023a} that if $ G $ is a complete multipartite graph that is not a star, then $ G $ is ZTF-dense. Thus, $G$ is ZTCF-dense.
\end{proof}

Most of the ZTCF-dense graphs identified so far satisfied $ Z(G) = Z_t(G) = Z_c(G) $. However, in general, ZTCF-dense graphs need not have equality among these parameters. The next result shows a family of ZTCF-dense graphs with $ Z(G) < Z_t(G) < Z_c(G)$. 

For $ k \ge 2 $, the \emph{diamond necklace} on $ k $ diamonds, denoted  $ N_k $, is obtained by taking $ k $ disjoint copies $ D_1, \dots, D_k $ of the diamond graph, where the vertices of $ D_i $ are $ \{a_i, b_i, c_i, d_i\} $ and the missing edge is $ a_ib_i $, and adding the edges $b_k a_1$ and $ b_i a_{i+1} $ for $ i \in [k-1] $. Diamond necklaces were introduced in~\cite{HeLo12}. A diamond necklace $ N_4 $ is shown in Figure~\ref{fig:forcing-abc-layout} (c). The following result about diamond necklaces is compiled from~\cite{DaHe18b} and~\cite{DaHePe2023a}.

\begin{thm}[\cite{DaHe18b},\cite{DaHePe2023a}]\label{t:necklace1}
If $ G $ is a diamond necklace of order $ n $, then $ G $ is ZTF-dense and
\[
Z(G) = \tfrac{1}{4}n + 2, \quad Z_t(G) = \tfrac{1}{2}n.
\]
\end{thm}

We now find the connected forcing number of diamond necklaces, and prove that diamond necklaces are also CF-dense.

\begin{thm}\label{thm:diamond-necklace}
If $ G $ is a diamond necklace of order $ n $, then $ G $ is ZTCF-dense and
\[
Z_c(G) = \tfrac{3}{4}n - 2.
\]
\end{thm}

\begin{proof}
Let $ G = N_k $ be a diamond necklace with order $ n = 4k $, where the vertices are labeled as in the definition of diamond necklace above. 

Let $S' = \bigcup_{i=1}^k \{a_i, b_i, x_i\}$, where for each $i$, $x_i\in\{c_i,d_i\}$. Let $S=S'\setminus\{a,b\}$, where $a\in \{a_1,\ldots,a_k\}$, $b\in \{b_1,\ldots,b_k\}$, and $a$ and $b$ are adjacent. By construction, $ |S| = 3k - 2 = \tfrac{3}{4}n - 2 $, and $ G[S] $ is connected. 

Suppose $a_i$ and $b_j$ are the elements of $\{a_1,\ldots,a_k,b_1,\ldots,b_k\}$ that are not in $S$, and suppose the vertices in $S$ are initially colored blue. Then,
\begin{enumerate}
    
\item[1)] $b_i$ can force the vertex in $\{c_i,d_i\}$ different from $x_i$; then, that vertex can force $a_i$.
\item[2)] $a_j$ can force the vertex in $\{c_j,d_j\}$ different from $x_j$; then, that vertex can force $b_j$.
\item[3)] For all $t\in [k]\setminus\{i,j\}$, $x_t$ can force the vertex in $\{c_t,d_t\}$ different from $x_t$.
\end{enumerate}
Therefore, $S$ is a connected forcing set of $G$, and $Z_c(G) \le |S| = \tfrac{3}{4}n - 2$. Let $S^*$ be a minimum connected forcing set of $G$. $S^*$ must contain at least one of $c_i$ and $d_i$ for each $i$, since otherwise $c_i$ and $ d_i $ can never be forced. A minimal connected subgraph of $G$ that contains at least one of $c_i$ and $d_i$ for each $i$ must connect them with a path that uses three  vertices per diamond, except for the two diamonds at the path's endpoints, which require only two vertices per diamond. Thus, $|S^*|\geq 3(k-2)+2+2=3k-2=\tfrac{3}{4}n - 2$, and so  $Z_c(G) = \tfrac{3}{4}n - 2 $. Moreover, since $S$ could be chosen to contain any vertex of $G$, it follows that $ G $ is CF-dense. Since by Theorem~\ref{t:necklace1}, $ G $ is ZTF-dense, we conclude that $ G $ is ZTCF-dense.
\end{proof}

\section{Connected Forcing in Trees}\label{sec:trees-and-unicyclic}
In this section, we characterize all CF-dense trees. We also give a formula to count all connected forcing sets (of any size) in trees. We will begin with recalling some definitions and known results. 

\begin{definition}
Let $G$ be a connected and non-path graph, and let $ v \in V $ be a vertex of degree at least three. A \emph{pendent path attached to $ v $} is a maximal induced path $ P \subseteq V $ such that $ G[P] $ is a connected component of $ G - v $, one endpoint of which is adjacent to $ v $. The neighbor of $ v $ in $ P $ is called the \emph{base} of the path, and we denote by $ p(v) $ the number of pendent paths attached to $ v $.
\end{definition}

\begin{definition}
If $\comp(G)$ denotes the number of connected components of $ G $, then
\begin{align*}
R_1(G) &= \{ v \in V : \comp(G - v) = 2,\; p(v) = 1 \}, \\
R_2(G) &= \{ v \in V : \comp(G - v) = 2,\; p(v) = 0 \}, \\
R_3(G) &= \{ v \in V : \comp(G - v) \ge 3 \}.
\end{align*}
\end{definition}
\begin{definition}
An \emph{$ \mathcal{M} $-set} of a graph $ G $ is any set of vertices that contains:
\begin{itemize}
    \item[1)] All vertices in $ R_2(G) $ and $ R_3(G) $, and
    \item[2)] For each vertex $ v \in V(G) $ with $d_G(v) \ge 3$, all-but-one of the bases of the pendent paths attached to $ v $.
\end{itemize}
\end{definition}

\begin{lem}[\cite{cf-complexity}]
\label{MR_lemma}
If $G$ is a connected graph different from a path and $R \subseteq V(G)$ is an arbitrary connected forcing set of $G$, then $R$ contains an $\mathcal{M}$-set.
\end{lem}
In view of Lemma \ref{MR_lemma}, $\mathcal{M}$-sets can be understood as sets of ``mandatory vertices" (up to the choice of bases of pendent paths), which appear in every connected forcing set.

\begin{thm}[\cite{cf-complexity}]
\label{tree_thm}
If $T$ is a tree different from a path, then any $\mathcal{M}$-set is a minimum connected forcing set of $T$.
\end{thm}

We now state the first main result of this section.

\begin{thm}\label{thm:tree-characterization}
If $T$ is a tree, then $T$ is CF-dense if and only if either:
\begin{enumerate}
    \item $ T \in \{ P_1, P_2 \} $, or
    \item Every support vertex of $ T $ is a strong support vertex.
\end{enumerate}
\end{thm}

\begin{proof}
By Lemma \ref{MR_lemma} and Theorem \ref{tree_thm}, a set $ S \subseteq V(T) $ is a minimum connected forcing set of a tree $ T $ if and only if $ S $ is an $ \mathcal{M} $-set. Thus, $ T $ is CF-dense if and only if every vertex of $ T $ lies in at least one $ \mathcal{M} $-set.

From the definition of $ \mathcal{M} $-sets, a vertex $ v \in V(T) $ lies in some minimum connected forcing set if and only if either $ v \in R_2(T) \cup R_3(T) $ or $ v $ is the base of a pendent path attached to a vertex with at least one other pendent path.

This implies that a vertex on a pendent path is included in a minimum connected forcing set only if it is the base of that path, and only if the vertex it attaches to has at least one other pendent path. In other words, all pendent paths must consist of a single vertex (the base), and be attached to a vertex that is incident to at least one additional pendent path. This condition is precisely equivalent to the statement that every support vertex is a strong support vertex.

Thus, a tree $ T $ is CF-dense if and only if either $ T \in \{ P_1, P_2 \} $, or every support vertex in $ T $ is a strong support vertex.
\end{proof}

\noindent It is easy to see that the condition in Theorem \ref{thm:tree-characterization} can be verified in linear time.
 
\noindent Next, we will conclude our study of connected forcing in trees by deriving a closed form expression for the number of distinct connected forcing sets of a tree. We begin by defining several combinatorial structures.

\begin{definition}
\label{definition_tuples}
Let $\mathcal{S}=\mathcal{S}(a;b_1,\ldots,b_k)$ be the set of $k$-tuples of positive integers whose sum is $a$ and whose $i^\text{th}$ element is at most $b_i$. The cardinality of $\mathcal{S}$ will be denoted by $s(a;b_1,\ldots, b_k)$, and the elements of $\mathcal{S}$ will be denoted by 
\[S(a;b_1,\ldots,b_k;1),\ldots, S(a;b_1,\ldots,b_k;s),\] where $s=s(a;b_1,\ldots, b_k)$. For $1\leq j\leq k$, $S(a;b_1,\ldots,b_k;i;j)$ will denote the $j^\text{th}$ element of the $k$-tuple $S(a;b_1,\ldots,b_k;i)$. When $a$ and $b_1,\ldots,b_k$ are clear from the context, we will write for short $S_i=S(a;b_1,\ldots,b_k;i)$ and $S_{i,j}=S(a;b_1,\ldots,b_k;i;j)$ for $1\leq i\leq s$, $1\leq j\leq k$. 
Define $\mathcal{S}'=\mathcal{S}'(a;b_1,\ldots,b_k)$ to be the set of $k$-tuples of nonnegative integers whose sum is $a$ and whose $i^\text{th}$ element is at most $b_i$, and define $s'$, $S'$, and $S'_{i,j}$ analogously to $s$, $S$, and $S_{i,j}$.
\end{definition}
\begin{example}
Consider the set $\mathcal{S}(5;2,2,6)$. The cardinality of this set is $4$, and its elements are $(1,1,3),(1,2,2),(2,1,2),(2,2,1)$. If $S_1=(1,1,3)$, then $S_{1,1}=1$, $S_{1,2}=1$, and $S_{1,3}=3$. Similarly, $\mathcal{S}'(5;2,2,6)=\{(0,0,5),(0,1,4),$ $(1,0,4),(1,1,3),$ $(2,0,3),(0,2,3),$ $(1,2,2),$ $(2,1,2),(2,2,1)\}$.
\end{example}

In addition to their integer partition definitions, there are several ways to interpret the sets $\mathcal{S}(a;b_1,\ldots,b_k)$ and $\mathcal{S}'(a;b_1,\ldots,b_k)$ in Definition \ref{definition_tuples}. For example, $\mathcal{S}$ represents the ways to place $a$ identical balls into $k$ distinct boxes, so that the $i^\text{th}$ box contains at least one and at most $b_i$ balls. Similarly, $\mathcal{S}'$ represents the ways to place $a$ identical balls into $k$ distinct boxes, so that the $i^\text{th}$ box contains at most $b_i$ balls. $\mathcal{S}$ is also the set of feasible solutions to the integer program $\{\min 0 :\sum_{i=1}^k x_i=a$, $x_1\leq b_1,\ldots,x_k\leq b_k$, $x\in \mathbb{N}\}$. The latter also gives a way to obtain all the elements of $\mathcal{S}$, i.e., solving the integer program while adding previous solutions as constraints. The elements of $\mathcal{S}$ and $\mathcal{S}'$ can also be found by a recursive combinatorial algorithm, e.g., using dynamic programming. Note that the sets $\mathcal{S}(a+k;b_1+1,\ldots,b_k+1)$ and $\mathcal{S}'(a;b_1,\ldots,b_k)$ are equivalent, in the sense that $s'=s$ and $S'_{i,j}=S_{i,j}-1$. This can easily be seen from the context of distributing balls to boxes: in $\mathcal{S}(a+k;b_1+1,\ldots,b_k+1)$ every box must contain a ball, so $k$ of the balls can be placed in the boxes right away, whereupon the capacity of box $i$ becomes $b_i$; then, the remaining $a$ balls can be distributed to the boxes without a lower bound, which is exactly the process described by $\mathcal{S}'(a;b_1,\ldots,b_k)$.

We now use Definition \ref{definition_tuples} to characterize the number of connected forcing sets of \emph{spiders}, which are graph composed of $k$ pendant paths attached to a central vertex $v$.

\begin{prop}
\label{prop_spider_cf}
If $G$ is a spider graph composed of $k$ pendant paths attached to a vertex $v$, where the $i^\text{th}$ pendant path has $b_i$ vertices, $1\leq i\leq k$, then

\begin{equation*}
z_c(G;j)=s(j-1;b_1,\ldots,b_k)+\sum_{\ell=1}^ks(j-1;b_1,\ldots b_{\ell-1},b_{\ell+1},\ldots,b_k).
\end{equation*}
\end{prop}
\begin{proof}
By Lemma \ref{MR_lemma} and Theorem \ref{tree_thm}, every connected set of vertices which contains $v$ and all-but-one neighbors of $v$ is a connected forcing set of $G$. Let $p_1,\ldots,p_k$ be the vertex sets of the pendant paths of $G$, where $|p_i|=b_i$. Then, the set of connected forcing sets of $G$ of size $j$ can be partitioned into the connected forcing sets of $G$ of size $j$ where every neighbor of $v$ is colored, and the connected forcing sets of $G$ of size $j$ where the neighbor of $v$ in $p_i$ is not colored (and hence also all other vertices in $p_i$ are not colored), $1\leq i\leq k$. 

The connected forcing sets of $G$ of size $j$ where every neighbor of $v$ is colored can be counted by $s(j-1;b_1,\ldots,b_k)$, since $v$ must be colored, which leaves $j-1$ of the remaining vertices of $G$ to be chosen, where at least one and at most $b_i$ vertices are chosen from $p_i$, $1\leq i\leq k$. The connected forcing sets of $G$ of size $j$ where no vertices of $p_\ell$ are chosen for some $\ell\in\{1,\ldots,k\}$ can be counted by $s(j-1;b_1,\ldots b_{\ell-1},b_{\ell+1},\ldots,b_k)$, since $v$ must be colored, which leaves $j-1$ of the remaining vertices of $G$ to be chosen, where at least one and at most $b_i$ vertices are chosen from $p_i$, $1\leq i\leq k$, $i\neq \ell$. Thus, the total number of connected forcing sets of $G$ of size $j$ is 
\begin{equation*}
z_c(G;j)=s(j-1;b_1,\ldots,b_k)+\sum_{\ell=1}^ks(j-1;b_1,\ldots b_{\ell-1},b_{\ell+1},\ldots,b_k).
\end{equation*}
\end{proof}

\noindent Using the previous results, we characterize the number of connected forcing sets of trees. 

\begin{thm}
\label{theorem_cf_poly_tree}
Let $T$ be a tree different from a path. Let $p_1,\ldots,p_k$ be the vertex sets of the pendant paths in $T$. Let $R_3^1$ be the set of vertices to which a single pendant path is attached, and $R_3^2=\{v_1\ldots,v_q\}$ be the set of vertices to which two or more pendant paths are attached. For $1\leq i\leq q$, let $I_j=\{i: p_i$ is a pendant path attached to $v_j\in R_3^2\}$; also, let $I_{q+1}=\{i: p_i$ is a pendant path attached to some vertex in $R_3^1\}$. For $1\leq j\leq q+1$, let $P_j=\bigcup_{i\in I_j} p_i$. Then, 
\begin{equation*}
z_c(T;d)=\sum_{i=1}^{s'}\left(\left(\prod_{j=1}^q s(|I_j|-1+S'_{i,j};[|p_\ell|:\ell\in I_j])\right)s'(S'_{i,q+1};[|p_\ell|:\ell\in I_{q+1}])\right),
\end{equation*}
where in the sum, $s'=s'(d-|\mathcal{M}|;|P_1|-|I_1|+1,\ldots,|P_q|-|I_q|+1,|P_{q+1}|)$, and for $1\leq i\leq q+1$, $S_{i,j}'=S'(d-|\mathcal{M}|;|P_1|-|I_1|+1,\ldots,|P_q|-|I_q|+1,|P_{q+1}|;i;j)$, and where $[|p_\ell|:\ell\in I_j]$ stands for the sequence of elements $|p_\ell|$ for all $\ell\in I_j$.

\end{thm}
\begin{proof}
By Lemma \ref{MR_lemma} and Theorem \ref{tree_thm}, every connected set of vertices which contains $R_2\cup R_3$ and all-but-one bases of pendant paths indexed by $I_j$, $1\leq j\leq q$, is a connected forcing set of $T$. Let $R$ be a connected forcing set of $T$ of size $d$. Then, $|R_2\cup R_3|$ of the vertices of $R$ are taken up by the vertices in $R_2\cup R_3$, and at least $|I_j|-1$ of the vertices of $R$ are taken up by vertices in $P_j$ for $1\leq j\leq q$. Thus, the remaining $d-|R_2\cup R_3|-(|I_1|-1)-\ldots-(|I_q|-1)=d-|\mathcal{M}|$ vertices of $R$ can be taken up by any of the sets $P_j$, $1\leq j\leq q+1$, as long as the number of vertices added to $P_j$ does not exceed $|P_j|-(|I_j|-1)$ for $1\leq j\leq q$ and does not exceed $|P_j|$ for $j=q+1$. By Definition \ref{definition_tuples}, the possible assignments of vertices according to these conditions are described by 
\begin{equation*}
\mathcal{S}'(d-|\mathcal{M}|;|P_1|-|I_1|+1,\ldots,|P_q|-|I_q|+1,|P_{q+1}|)=\{S'_1,\ldots,S'_{s'}\}.
\end{equation*}

Thus, the set of connected forcing sets of $T$ of size $d$ can be partitioned into $s'$ parts, where the $i^\text{th}$ part consists of the connected forcing sets of $G$ of size $d$ where $|I_j|-1+S_{i,j}'$ of the colored vertices are in $P_j$, $1\leq j\leq q$, and $S_{i,q+1}'$ of the colored vertices are in $P_{q+1}$. Moreover, when $|I_j|-1+S_{i,j}'$ vertices are allotted to $P_j$, $1\leq j\leq q$, they must be distributed among the pendant paths indexed by $I_j$ in such a way that all or all-but-one of them have at least one colored vertex, and so that the pendant path with vertex set $p_\ell$ does not receive more than $|p_\ell|$ vertices, for any $\ell\in I_j$. Similarly, when $S_{i,q+1}'$ vertices are allotted to $P_{q+1}$, they can be distributed among the pendant paths indexed by $I_{q+1}$ without a lower bound, as long as path $p_\ell$ does not receive more than $|p_\ell|$ vertices, for any $\ell\in I_{q+1}$.

Finally, when some number of colored vertices is assigned to the sets $P_1,\ldots,P_{q+1}$, one can distribute the colored vertices allotted to $P_j$ to the specific pendant paths in $P_j$ independently for each $P_j$. Thus, the number of connected forcing sets of $G$ of size $d$ where $|I_j|-1+S_{i,j}$ of the colored vertices are in $P_j$, $1\leq j\leq q$, and $S_{i,q+1}'$ of the colored vertices are in $P_{q+1}$, is equal to the product of the number of ways to assign $|I_j|-1+S_{i,j}$ and $S_{i,q+1}'$ colored vertices, respectively, to the sets $P_j$, $1\leq j\leq q$ and $P_{q+1}$. In turn, by a similar argument as in Proposition \ref{prop_spider_cf}, the number of ways to assign $|I_j|-1+S_{i,j}'$ vertices to $P_j$, $1\leq j\leq q$, is $s(|I_j|-1+S_{i,j}';[|p_\ell|:\ell\in I_j])$, where $[|p_\ell|:\ell\in I_j]$ stands for the sequence of elements $|p_\ell|$ for all $\ell\in I_j$. Similarly, the number of ways to assign $S_{i,q+1}'$ vertices to $P_{q+1}$ is $s'(S_{i,q+1}';[|p_\ell|:\ell\in I_{q+1}])$. Thus, multiplying for $1\leq j\leq q+1$ and summing over $1\leq i\leq s'=s'(d-|\mathcal{M}|;|P_1|-|I_1|+1,\ldots,|P_q|-|I_q|+1,|P_{q+1}|)$, we conclude that the number of connected forcing sets of $G$ of size $d$ is

\begin{equation*}
z_c(T;d)=\sum_{i=1}^{s'}\left(\left(\prod_{j=1}^q s(|I_j|-1+S'_{i,j};[|p_\ell|:\ell\in I_j])\right)s'(S'_{i,q+1};[|p_\ell|:\ell\in I_{q+1}])\right).
\end{equation*}
\end{proof}

\begin{figure}[htb]
\begin{center}
\begin{tikzpicture}[scale=1,style=thick,x=1cm,y=1cm]
\def\vr{2.5pt} 

\path (0, 1) coordinate (v1) node[above] {$v_1$};
\path (1, 1) coordinate (v2);
\path (2, 1) coordinate (v3);
\path (3, 1) coordinate (v4) node[above] {$v_2$};

\path (-1, 0) coordinate (v11) node[below] {$p_1$};
\path (0, 0) coordinate (v12) node[below] {$p_2$};

\path (1, 0) coordinate (v21);
\path (1, -1) coordinate (v211) node[below] {$p_3$};

\path (2, 0) coordinate (v31) node[below] {$p_4$};

\path (3, 0) coordinate (v41) node[below] {$p_5$};
\path (4, 0) coordinate (v42);
\path (4, -1) coordinate (v421) node[below] {$p_6$};

\draw (v1) -- (v2);
\draw (v2) -- (v3);
\draw (v3) -- (v4);
\draw (v1) -- (v11);
\draw (v1) -- (v12);

\draw (v2) -- (v21);
\draw (v21) -- (v211);

\draw (v3) -- (v31);

\draw (v4) -- (v41);
\draw (v4) -- (v42);
\draw (v42) -- (v421);

\draw (v1) [fill=white] circle (\vr);
\draw (v2) [fill=white] circle (\vr);
\draw (v3) [fill=white] circle (\vr);
\draw (v4) [fill=white] circle (\vr);

\draw (v11) [fill=white] circle (\vr);
\draw (v12) [fill=white] circle (\vr);

\draw (v21) [fill=white] circle (\vr);
\draw (v211) [fill=white] circle (\vr);

\draw (v31) [fill=white] circle (\vr);

\draw (v41) [fill=white] circle (\vr);
\draw (v42) [fill=white] circle (\vr);
\draw (v421) [fill=white] circle (\vr);

\end{tikzpicture}
\caption[A tree for which $z_c(T;8)$ is computed]{A tree $T$ for which $z_c(T;8)$ is computed in Example \ref{example_tree_cf_poly}.}
\label{fig:cf_example_tree}
\end{center}
\end{figure}
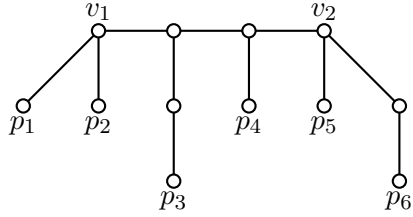

\begin{example}
\label{example_tree_cf_poly}
In this example, we will apply Theorem \ref{theorem_cf_poly_tree} to enumerate the connected forcing sets of size $d=8$ for the tree $T$ shown in Figure \ref{fig:cf_example_tree}. 

The pendant paths of $T$ are labeled $p_1,\ldots,p_6$. First, note that $Z_c(T)=|\mathcal{M}(T)|=6$ and $R_3^2=\{v_1,v_2\}$. Thus, $I_1=\{1,2\}$, $I_2=\{5,6\}$, $I_3=\{3,4\}$, and $P_1=p_1\cup p_2$, $P_2=p_5\cup p_6$, $P_3=p_3\cup p_4$. The total number of vertices in these sets is $|P_1|=2$, $|P_2|=3$, $|P_3|=3$. Using these values, we compute
\begin{eqnarray*}
&&\mathcal{S}'(d-|\mathcal{M}|;|P_1|-|I_1|+1,|P_2|-|I_2|+1,|P_3|)=\mathcal{S}'(2;1,2,3)=\\
&&=\{(1,1,0),(0,2,0),(1,0,1),(0,1,1),(0,0,2)\}=\{S_1',S_2',S_3',S_4',S_5'\}
\end{eqnarray*}
We now apply the formula for $z_c(T;d)$ given in Theorem \ref{theorem_cf_poly_tree}.
\begin{eqnarray*}
z_c(T;d)&=&\sum_{i=1}^{s'}\left(\left(\prod_{j=1}^q s(|I_j|-1+S'_{i,j};[|p_\ell|:\ell\in I_j])\right)s'(S'_{i,q+1};[|p_\ell|:\ell\in I_{q+1}])\right)\\
z_c(T;8)&=&\sum_{i=1}^{5}\Big(s(2-1+S'_{i,1};|p_1|,|p_2|)s(2-1+S'_{i,2};|p_5|,|p_6|)\Big)s'(S'_{i,3};|p_3|,|p_4|)\\
&=&\sum_{i=1}^{5}s(1+S'_{i,1};1,1)s(1+S'_{i,2};1,2)s'(S'_{i,3};2,1)\\
&=&s(2;1,1)s(2;2,1)s'(0;1,2)+s(1;1,1)s(3;2,1)s'(0;1,2)+\\
&&s(2;1,1)s(1;2,1)s'(1;1,2)+s(1;1,1)s(2;2,1)s'(1;1,2)+\\
&&s(1;1,1)s(1;2,1)s'(2;1,2)\\
&=&1\cdot 2\cdot 1+2\cdot 1\cdot 1+1\cdot 2\cdot 2+2\cdot 2\cdot 2+2\cdot 2\cdot 2=24
\end{eqnarray*}
Thus, there are 24 different connected forcing sets of size 8. These are shown in Figure~\ref{fig:cf_example_tree2}, where the panel showing the connected forcing sets corresponding to $S_1'$, $S_2'$, $S_3'$, $S_4'$, and $S_5'$, respectively, is the left-top, left-bottom, middle-left, middle-right, and right panel.

\newcommand{\TreeWithColors}[1]{
\begin{tikzpicture}[scale=0.35, style=thick, x=1cm, y=1cm]
\def\vr{4.5pt}

\path (0, 0) coordinate (v1);
\path (1, 0) coordinate (v2);
\path (2, 0) coordinate (v3);
\path (3, 0) coordinate (v4);
\path (-1, -1) coordinate (v11);
\path (0, -1) coordinate (v12);
\path (1, -1) coordinate (v21);
\path (1, -2) coordinate (v211);
\path (2, -1) coordinate (v31);
\path (3, -1) coordinate (v41);
\path (4, -1) coordinate (v42);
\path (4, -2) coordinate (v421);

\draw (v1) -- (v2);
\draw (v2) -- (v3);
\draw (v3) -- (v4);
\draw (v1) -- (v11);
\draw (v1) -- (v12);
\draw (v2) -- (v21);
\draw (v21) -- (v211);
\draw (v3) -- (v31);
\draw (v4) -- (v41);
\draw (v4) -- (v42);
\draw (v42) -- (v421);

\foreach \v in {v1,v2,v3,v4,v11,v12,v21,v211,v31,v41,v42,v421}
  \draw (\v) [fill=white] circle (\vr);

\foreach \v in {#1}
  \draw (\v) [fill=blue] circle (\vr);

\end{tikzpicture}
}

\begin{figure}[htb]
\centering
\includegraphics[width=1\textwidth]{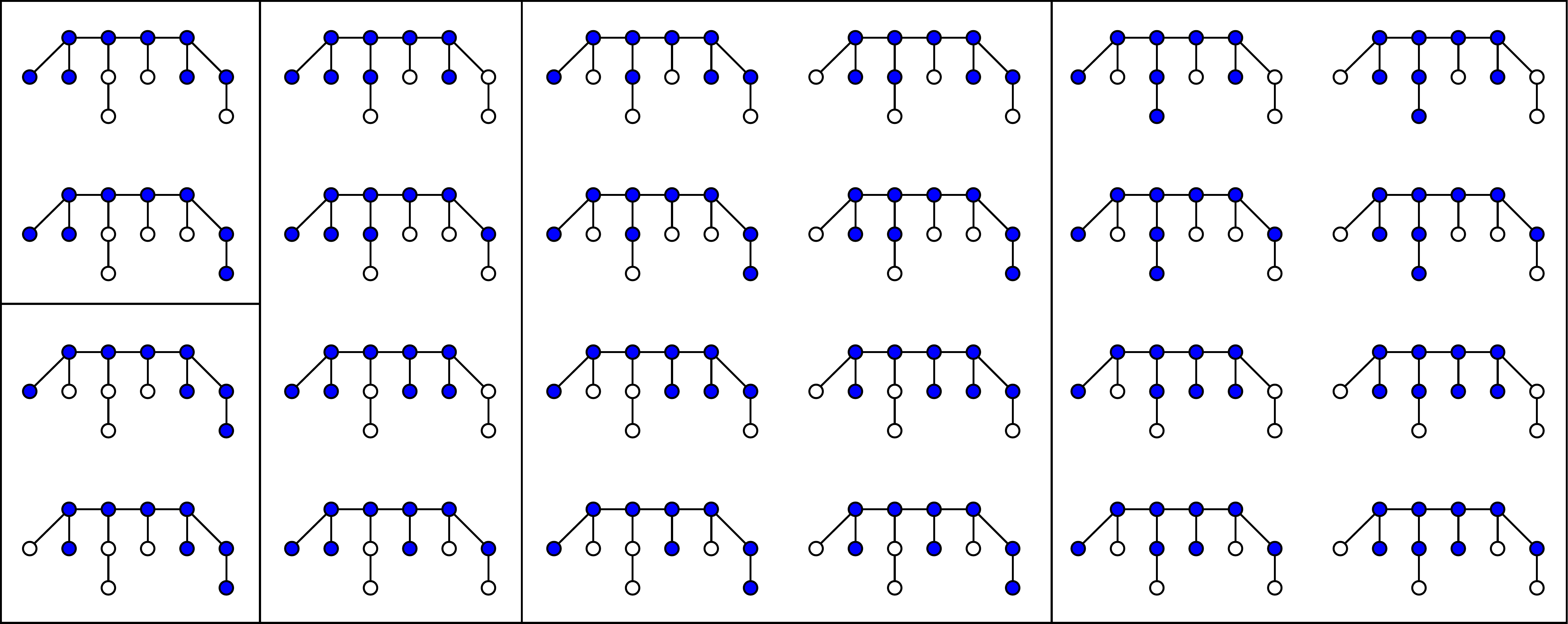}

\caption{All connected forcing sets of size 8 for the tree $T$.}
\label{fig:cf_example_tree2}
\end{figure}
\end{example}

\section{Graph Operations and CF-Density}
\label{sec:cartesian-alt}
In this section we explore the effect of Cartesian products, coronas, and joins on the total forcing and connected forcing numbers of graphs. We then derive conditions under which CF-density is preserved under these operations. This allows us to generate new families of ZTCF-dense graphs in addition to the ones identified in Sections \ref{sec:basics} and \ref{sec:trees-and-unicyclic}. 
\subsection{Cartesian products}

The \emph{Cartesian product} of graphs $G$ and $H$, denoted $G \cprod H$, is the graph with vertex set $V(G)\times V(H)$, where vertices $(g_1,h_1)$ and $(g_2,h_2) $ are adjacent if and only if either $ g_1 = g_2 $ and $ h_1h_2 \in E(H) $, or $ h_1 = h_2 $ and $ g_1g_2 \in E(G) $. We begin with a lemma extending a result from \cite{AIM-Workshop} about the zero forcing number of the Cartesian product of two graphs.

\begin{lem}\label{zfbox}
If $B$ is a zero forcing set of $G$, then $B'=\{(g,h):g \in B, h \in V(H)\}$ is a zero forcing set of $G \cprod H$.
\end{lem}

\begin{proof}
We verify that $ B' $ is a zero forcing set. Let $ B = B_0, B_1, B_2, \dots $ be a chronological zero forcing sequence for $ B $ on $ G $, and define corresponding sets $ B_k' = \{ (g,h) : g \in B_k, h \in V(H) \} $. We show by induction on $ k $ that all vertices in $ B_k' $ are blue at time step $ k $ in the forcing process on $ G \cprod H $ starting from $ B_0' $.

The base case $ k = 0 $ is immediate. Assume the inductive hypothesis holds for step $ k $, and let $ v \in B_{k+1} \setminus B_k $. Then there exists a vertex $ u \in B_k $ such that $ v $ is the unique white neighbor of $ u $ in $ G $. For each $ h \in V(H) $, $ (u,h) \in B_k' $ has neighbor $ (v,h) $ in $ G \cprod H $, and $ (v,h) $ is the unique white neighbor of $ (u,h) $. Thus, each $ (u,h) $ forces $ (v,h) $ at time step $ k+1 $. It follows that $ B_{k+1}' $ is entirely blue. Therefore, $ B' $ is a zero forcing set.
\end{proof}

Using Lemma \ref{zfbox}, we derive the following sufficient condition about when Cartesian products preserve ZF-density.

\begin{thm}\label{t:cart_z}
Let $ G $ and $ H $ be graphs. If $ G $ is ZF-dense and
\[
Z(G \cprod H) = Z(G) \cdot |V(H)|,
\]
then $ G \cprod H $ is ZF-dense.
\end{thm}

\begin{proof}
Fix an arbitrary vertex $ (g_0, h_0) \in V(G \cprod H) $. Since $ G $ is ZF-dense, there exists a minimum zero forcing set $ B \subseteq V(G) $ such that $ g_0 \in B $, with $ |B| = Z(G) $. 
 Define the set $B'$ as in Lemma \ref{zfbox}, and note that $B'$ is a zero forcing set of $G \cprod H$.  Finally, since $ |B'| = |B| \cdot |V(H)| = Z(G) \cdot |V(H)| = Z(G \cprod H) $, $ B' $ is a minimum zero forcing set. Since $ (g_0, h_0) \in V(G \cprod H) $ was arbitrary, it follows that every vertex belongs to some minimum zero forcing set, and $ G \cprod H $ is ZF-dense.
\end{proof}

Next, we derive a sufficient condition about when Cartesian products preserve TF-density.

\begin{lem}\label{l:cart_t}
Let $G$ be a graph and $H$ be a graph with no isolated vertices.  If $B$ is a zero forcing set of $G$, then $B'=\{(g,h):g \in B, h \in V(H)\}$ is a total forcing set of $G \cprod H$ and $Z_t(G \cprod  H) \leq Z(G)\cdot|V(H)|.$
\end{lem}

\begin{proof}
Let $B \subseteq V(G)$ be a minimum zero forcing set of $G$.  Next note that by Lemma \ref{zfbox}, $B'$ is a zero forcing set of $G \cprod H$.  Finally, since $H$ has no isolated vertices, $(G \cprod H)[B']$ contains no isolated vertices.  Thus $B'$ is a total forcing set of $G \cprod H$, and $Z_t(G \cprod  H) \leq |B'|=Z(G)\cdot |V(H)|$.
\end{proof}

\begin{cor}
If $G$ and $H$ are graphs with no isolated vertices, then
\[Z_t(G \cprod H) \leq \min\{Z(G) \cdot |V(H)|, Z(H) \cdot |V(G)|\}.\]
\end{cor}

\begin{thm}\label{t:cart_t}
Let $ G $ be a ZF-dense graph and $ H $ be a graph with no isolated vertices. If
\[
Z_t(G \cprod H) = Z(G) \cdot |V(H)|,
\]
then $ G \cprod H $ is TF-dense.
\end{thm}

\begin{proof}
Fix an arbitrary vertex $ (g_0, h_0) \in V(G \cprod H) $. Since $ G $ is ZF-dense, there exists a minimum zero forcing set $ B \subseteq V(G) $ such that $ g_0 \in B $, with $ |B| = Z(G) $. 
 Define the set $B'$ as in Lemma \ref{l:cart_t}, and note that $B'$ is a total forcing set of $G \cprod H$. 
 Finally, since $ |B'| = |B| \cdot |V(H)| = Z(G) \cdot |V(H)| = Z_t(G \cprod H) $, $ B' $ is a minimum total forcing set. Since $ (g_0, h_0) \in V(G \cprod H) $ was arbitrary, it follows that every vertex belongs to some minimum total forcing set, and $ G \cprod H $ is TF-dense.
\end{proof}

Finally, we derive an analogous sufficient condition about CF-density.

\begin{lem}\label{l:cart_c}
Let $G$ and $H$ be connected graphs.  If $B$ is a connected forcing set of $G$, then $B'=\{(g,h):g \in B, h \in V(H)\}$ is a connected forcing set of $G \cprod H$ and
\[Z_c(G \cprod  H) \leq \min\{Z_c(G)\cdot|V(H)|, Z_c(H)\cdot|V(G)|\}.\]
\end{lem}

\begin{proof}
Let $B \subseteq V(G)$ be a minimum connected forcing set of $G$.  Now note that by Lemma \ref{zfbox}, $B'$ is a zero forcing set of $G \cprod H$.  Next we show that $(G \cprod H)[B']$ is connected and thus $B'$ is a connected forcing set of $G \cprod H$.  Fix $ (g_1, h_1), (g_2, h_2) \in B' $. Since $ G[B] $ is connected, there exists a $ g_1 $--$ g_2 $ path in $ G[B] $. Holding $ h_1 $ fixed, this lifts to a path from $ (g_1, h_1) $ to $ (g_2, h_1) $ in $ (G \cprod H)[B'] $. Similarly, since $ H $ is connected, there is a path from $ (g_2, h_1) $ to $ (g_2, h_2) $. Thus, $ (G \cprod H)[B'] $ is connected, $B'$ is a connected forcing set of $G \cprod H$, and $Z_c(G \cprod  H) \leq |B'|=Z_c(G)\cdot |V(H)|$.
\end{proof}

\begin{thm}\label{t:cart_c-alt}
Let $ G $ and $ H $ be connected graphs. If $ G $ CF-dense and 
\[
Z_c(G \cprod H) = Z_c(G) \cdot |V(H)|,
\]
then $ G \cprod H $ is CF-dense.
\end{thm}

\begin{proof}
Fix an arbitrary vertex $ (g_0, h_0) \in V(G \cprod H) $. Since $ G $ is CF-dense, there exists a minimum connected forcing set $ B \subseteq V(G) $ such that $ g_0 \in B $, with $ |B| = Z_c(G) $. Define the set $B'$ as in Lemma \ref{l:cart_c}, and note that $B'$ is a connected forcing set of $G \cprod H$.  Finally, since $ |B'| = |B| \cdot |V(H)| = Z_c(G) \cdot |V(H)| = Z_c(G \cprod H) $, $ B' $ is a minimum connected forcing set. Since $ (g_0, h_0) \in V(G \cprod H) $ was arbitrary, it follows that every vertex belongs to some minimum connected forcing set, and $ G \cprod H $ is CF-dense.
\end{proof}

Theorems \ref{t:cart_z}, \ref{t:cart_t}, and \ref{t:cart_c-alt}  yield the following corollary.

\begin{cor}\label{c:cart}
Let $G$ and $H$ be connected graphs each on at least two vertices, with $G$ ZCF-dense.  If $Z_c(G \cprod H)=Z_c(G)\cdot|V(H)|$ and $Z_t(G \cprod  H)=Z(G \cprod H)=Z(G)\cdot|V(H)|$, then $G \cprod H$ is ZTCF-dense.
\end{cor}

Examples of graphs satisfying the conditions of Corollary \ref{c:cart} (and thus also witnessing the sharpness of the bounds in Lemmas \ref{l:cart_t} and \ref{l:cart_c}) include $P_r \cprod P_s$ for $r,s \geq 2$ and $P_r \cprod C_s$ for $1 \leq r \leq \frac{1}{2}s$.  

\subsection{Graph joins}

The \emph{join} of graphs $G$ and $H$, denoted $G \vee H$, is the graph with vertex set $V(G \vee H)=V(G) \cup V(H)$ and edge set $E(G \vee H)=E(G) \cup E(H) \cup \{gh: g \in V(G), h \in V(H)\}$.  The following results concerning the zero forcing and total forcing numbers of the join of two graphs were introduced in \cite{join} and \cite{DaHePe2023a} respectively. 

\begin{lem}[\cite{join}]\label{l:join:z}
If $G$ and $H$ are graphs, with each either connected or isolate free, then
\[
Z(G \vee H)=\min\{Z(G)+|V(H)|,Z(H)+|V(G)|\}.
\]
\end{lem}

\begin{lem}[\cite{DaHePe2023a}]\label{l:join:t}
If $G$ and $H$ are graphs, with each either connected or isolate free, then
\[
Z_t(G \vee H)=\min\{Z(G)+|V(H)|,Z(H)+|V(G)|\}.
\]
\end{lem}

We now extend these results to connected forcing. 
\begin{lem}\label{l:join:c}
If $G$ and $H$ are graphs, with each either connected or isolate free, then 
\begin{enumerate}
\item For each zero forcing set $B_G$ of $G$ and each zero forcing set $B_H$ of $H$, $B_G \cup V(H)$ and $B_H \cup V(H)$ are zero forcing sets of $G \vee H$.
\item $Z_c(G \vee H)=\min\{Z(G)+|V(H)|,Z(H)+|V(G)|\}$.
\end{enumerate}
\end{lem}

\begin{proof}
First let $B_G$ be a zero forcing set of $G$ and consider the set $B=B_G \cup V(H)$.  Since there are no white vertices outside of $G$, forcing in $G \vee H$ can proceed from $B$ exactly as forcing in $G$ proceeds from $B_G$, and thus $B_G \cup V(H)$ is a zero forcing set of $G \vee H$.  Finally, since in $G \vee H$ every vertex in $G$ is adjacent every vertex in $H$, each pair of vertices in $(G \vee H)[B]$ is either adjacent or distance 2.  Thus $(G \vee H)[B]$ is connected, and so $B$ is a connected forcing set of $G \vee H$.  A similar argument shows that for each zero forcing set $B_H$ of $H$, $B_H \cup V(G)$ is a zero forcing set of $G \vee H$.

Now let $B_G$ be a minimum zero forcing set of $G$ and suppose without loss of generality that $Z(G)+|V(H)|=\min\{Z(G)+|V(H)|,Z(H)+|V(G)|\}$.    So $Z_c(G \vee H) \leq |B|=Z(G)+|V(H)|=\min\{Z(G)+|V(H)|,Z(H)+|V(G)|\}$.  Finally, we obtain equality since $Z_c(G \vee H) \geq Z(G \vee H)=\min\{Z(G)+|V(H)|,Z(H)+|V(G)|\}$, by Lemma \ref{l:join:z}.
\end{proof}

\begin{thm}\label{joinZTC}
Let $G$ and $H$ be graphs of order $n(G)$ and $n(H)$ respectively, with each either connected or isolate free.  If $G$ is ZF-dense and $n(H)-Z(H) \leq n(G)-Z(G)$, then $G \vee H$ is ZTCF-dense.
\end{thm}

\begin{proof}
First note that it was shown in \cite{DaHePe2023a}, that under these conditions $G \vee H$ is ZTF-dense, so it remains to show that $G \vee H$ is also CF-dense.  Next note that since $n(H)-Z(H) \leq n(G)-Z(G)$, it follows from Lemma \ref{l:join:c} that $Z(G \vee H)=\min\{Z(G)+|V(H)|,Z(H)+|V(G)|\}=Z(G) + |V(H)|$.  So if $B_G$ is a minimum zero forcing set of $G$ then $B_G \cup V(H)$ is a minimum zero forcing set of $G \vee H$.

Now let $v \in V(G \vee H)$ be arbitrary.  If $v \in V(H)$, then for any minimum zero forcing set $B_G$ of $G$, $v \in B_G \cup V(H)$.  So suppose that $v \in V(G)$.  Since $G$ is ZF-dense, there exists a minimum zero forcing set $B_G$ of $G$ such that $v \in B_G$, and thus $v \in B_G \cup V(H)$.  Finally, since $v \in V(G \vee H)$ was arbitrary, it follows that every vertex of $G \vee H$ belongs to some minimum connected forcing set, and $G \vee H$ is CF-dense. 
\end{proof}

An example of a graph satisfying the conditions of Theorem \ref{joinZTC} is $C_n \vee H$, where $H$ is any graph on at most $n-1$ vertices.  The following result concerning the addition of dominating vertices is an immediate corollary of Theorem \ref{joinZTC}.

\begin{cor}
Let $G$ be a graph and $v \in V(G)$ be a dominating vertex of $G$.  If $G-v$ is ZF-dense and either connected or isolate free, then $G$ is ZTCF-dense.
\end{cor}

\subsection{Coronas}

The \emph{corona} of $G$ with $H$, denoted $G\circ H$, is obtained by taking one copy of $G$ and $|V(G)|$ copies of $H$ and joining the $i$th vertex of $G$ to every vertex in the $i$th copy of $H$.  The following result first appeared in \cite{ZGcircKs} (with a small correction to the result made in \cite{cartfrac}).  

\begin{prop}[\cite{cartfrac, ZGcircKs}]\label{coronaval}
If $G$ and $H$ be graphs of order $q$ and $r \geq 2$ and $H$ is isolate-free, then $Z(G \circ H)=qZ(H)+Z(G)$.
\end{prop}

We now extend this result in a way that will be helpful for discussing forcing density.

\begin{lem}\label{lem:coronasets}
Let $G$ and $H$ be graphs such that $H$ has no isolated vertices.  If $B_G$ is a zero forcing set of $G$, $B_H$ is a zero forcing set of $H$, and $B_H^i$ is the copy of $B_H$ in the $i$th copy of $H$, then $B=B_G \cup (\bigcup_{i=1}^q B_H^i)$ is a zero forcing set of $G \circ H$.  In addition, if $B_G$ and $B_H$ are both minimum, then $B$ is minimum. 
\end{lem}

\begin{proof}
First construct $B$ as above.  Allowing the forcing process in $G \circ H$ to begin at $B$, forces will alternate between those in $G$ $u_i \rightarrow u_j$ when every vertex in the $i$th copy of $H$ is blue and those in copies of $H$ when the vertex $u_i$ is blue, and thus $B$ is a zero forcing set of $G \circ H$.  If in addition, $B_G$ and $B_H$ are minimum, then $|B|=qZ(H)+Z(G)=Z(G \circ H)$.  So by Proposition \ref{coronaval}, $B$ is a minimum zero forcing set of $G \circ H$.   
\end{proof}

\begin{lem}\label{lem:coronaforce}
Let $G$ and $H$ be graphs such that $H$ has no isolated vertices.    Enumerate the vertices of $G$, $\{u_i\}_{i=1}^q$, and for each $i \in [q]$, let $H_i$ be the $i$th copy of $H$.  If $B'$ is any zero forcing set of $G \circ H$, then it can be assumed that for each $i \in [q]$, $u_i$ forces no vertices in $H_i$ during the forcing process beginning at $B'$. 
\end{lem}

\begin{proof}
Let $B'$ be a zero forcing set of $G \circ H$, and for each $i \in [q]$, let $H_i'=(G \circ H)[V(H_i) \cup \{u_i\}]$.  Since each $u_i$ is a dominating vertex of $H_i'$, $u_i$ will have more than one white neighbor until only one vertex $w$ of $H_i$ is white.  Since $H$ has no isolated vertices, $w$ has a neighbor $v$ in $H_i$; and at this point in the forcing process, since every vertex in $H_i'$ is blue except for $w$, $v$ can force $w$.  So it can be assumed that during the forcing process beginning at $B'$, $u_i$ does not force any vertices in $H_i$.  
\end{proof}

\begin{lem}\label{totalcoronaval}
Let $G$ be a graph and $H$ be a graph with no isolated vertices.
\begin{itemize}
\item If $Z_t(H)=Z(H)$, then $Z_t(G \circ H)=|V(G)|\cdot Z(H)+Z(G)$.
\item If $Z_t(H)>Z(H)$, then $Z_t(G \circ H)=|V(G)|\cdot Z(H)+|V(G)|$.
\end{itemize}
\end{lem}

\begin{proof}
Begin by enumerating the vertices of $G$, $\{u_i\}_{i=1}^q$.  Next, for each $i \in [q]$, let $H_i'=(G \circ H)[V(H_i) \cup \{u_i\}]$.

First suppose that $Z_t(H)=Z(H)$ and note that by Proposition \ref{coronaval}, $Z_t(G \circ H) \geq Z(G \circ H)=|V(G)| \cdot Z(H)+Z(G)$.  If $B_G$ is a minimum total forcing set of $G$, $B_H$ is a minimum zero forcing set of $H$, and $B_H^i$ is the copy of $B_H$ in the $i$th copy of $H$, then by Lemma \ref{lem:coronasets} and since total forcing sets are zero forcing sets, $B=B_G \cup (\bigcup_{i=1}^q B_H^i)$ is a zero forcing set of $G \circ H$.  Furthermore, for each $i \in [q]$, since $B_H^i$ is a total forcing set, each vertex in $B_H^i$ has a neighbor in $B_H^i$, and since each vertex $u_i$ is a dominating vertex of $H_i'$, each $u_i \in B_G$ has neighbors in $B_H^i$.  Thus $B$ is a total forcing set of $G \circ H$ and since $Z_t(H)=Z(H)$, $Z_t(G \circ H) \leq |V(G)| \cdot Z_t(H)+Z(G)=|V(G)| \cdot Z(H)+Z(G)$.

Next suppose that $Z_t(H)>Z(H)$ and let $B=V(G) \cup \left(\bigcup_{i=1}^{q}B_H^i\right)$ such that $B_H$ is a zero forcing set of $H$ and for each $i \in [q]$, $B_H^i$ is the copy of $B_H$ in $H$.  Clearly, $V(G)$ is a zero forcing set of $G$, so by Lemma \ref{lem:coronasets}, $B$ is a zero forcing set of $G \circ H$.  Furthermore, since each $u_i$ is a dominating vertex of $H_i'$, each vertex in $V(G)$ has a neighbor in $B$ and each vertex in $B_H^i$ has a neighbor in $B$.  Thus $B$ is a total forcing set of $G \circ H$, and $Z_t(G \circ H) \leq |V(G)| \cdot Z(H)+|V(G)|$.  

Finally, let $B$ be an arbitrary minimum total forcing set of $G \circ H$.  Since each vertex $u_i$ is a dominating vertex of $H_i'$ and $H$ has no isolated vertices, $Z(H_i')=Z(H)+1$.  Furthermore, since $u_i$ is the only vertex in $H_i'$ with neighbors outside of $H_i'$, at most one vertex in $H_i'$ can be forced by a vertex outside of $H_i'$, and so $|B \cap V(H_i')| \geq Z(H_i')-1=Z(H)$.  Now suppose, by way of contradiction, that $|B \cap V(H_i')|=Z(H)$.  First consider the case where $u_i \in B$.  Since $u_i \in B$, $|B \cap V(H_i)| < Z(H_i)$ cannot be a zero forcing set of $H_i$.  In turn, by Lemma \ref{lem:coronaforce} it can be assumed that $u_i$ does not force any vertices in $H_i'$, so $B$ cannot be a zero forcing set of $G \circ H$.  Now consider the case where $u_i \not \in B$, again by Lemma \ref{lem:coronaforce}, it can be assumed that $u_i$ forces no vertices in $H_i'$, so it must be the case that $B \cap V(H_i)$ is a minimum zero forcing set of $H_i$.  However, $Z_t(H_i)>Z(H_i)$, so there must exist vertices in $B$ which do not have neighbors in $B$ and $B$ cannot be a total forcing set of $G \circ H$.  Thus, for each $i \in [q]$, $B \cap V(H_i') \geq Z(H)+1$.  So $Z_t(G \circ H)=|B|\geq |V(G)|\cdot Z(H)+|V(G)|$.
\end{proof}

\begin{lem}\label{connectcoronaval}
If $G$ is a connected graph and $H$ is an isolate-free graph, then 
\[Z_c(G \circ H)=|V(G)| \cdot Z(H) + |V(G)|.\] 
\end{lem}

\begin{proof}
Begin by enumerating the vertices of $G$, $\{u_i\}_{i=1}^q$.  Next, for each $i \in [q]$, let $H_i'=(G \circ H)[V(H_i) \cup \{u_i\}]$.

First let $B=V(G) \cup \left(\bigcup_{i=1}^{q}B_H^i\right)$, where $B_H$ is a zero forcing set of $H$, and for each $i \in [q]$, $B_H^i$ is the copy of $B_H$ in $H_i$.  Since $V(G)$ is a zero forcing set of $G$ and $B_H$ is a zero forcing set of $H$, by Lemma \ref{lem:coronasets}, $B$ is a zero forcing set of $G \circ H$.  Furthermore, since every vertex in $B_H^i$ is adjacent to the corresponding $u_i$ and $G$ is a connected graph, it follows that $(G \circ H)[B]$ is connected, and as a result $B$ is a connected forcing set of $G \circ H$.  Thus, $Z_c(G \circ H) \leq |B|=|V(G)| \cdot Z(H)+|V(G)|$.  

Now let $B$ be a minimum connected forcing set of $G \circ H$.  By Lemma \ref{lem:coronaforce}, it can be assumed that $u_i$ forces no vertices in $H_i$, and thus since $u_i$ is the only vertex of $H_i'$ with neighbors outside of $H_i'$, $B \cap V(H_i)$ must be a zero forcing set of $H_i$.  Furthermore, since this means that each $H_i$ contains vertices in $B$, but no vertex in $H_i$ is adjacent a vertex in $H_j$ for $i \neq j$, for $(G \circ H)[B]$ to be connected, every vertex in $V(G)$ must be contained in $B$.  Thus $Z_c(G \circ H) \geq |V(G)|\cdot Z(H)+|V(G)|$.
\end{proof}

Using the lemmas above, we now derive sufficient conditions for the corona to preserve ZF-density, TF-density, and CF-density. 

\begin{thm}
\label{thmcortfzf}
Let $G$ and $H$ be graphs, with $H$ isolate-free, and suppose $G$ is ZF-dense.
\begin{itemize}
\item If $H$ is ZF-dense, then $G \circ H$ is ZF-dense.
\item If $H$ is TF-dense and $Z_t(H)=Z(H)$, then $G \circ H$ is TF-dense.
\item If $H$ is ZF-dense and $Z_t(H)>Z(H)$, then $G \circ H$ is TF-dense.
\end{itemize}
\end{thm}

\begin{proof}
Begin by enumerating the vertices of $G$, $\{u_i\}_{i=1}^q$.  Next, for each $i \in [q]$, let $H_i'=(G \circ H)[V(H_i) \cup \{u_i\}]$.  Now let $v \in V(G \circ H)$ be arbitrary.

First suppose that $H$ is ZF-dense.  If $v \in V(G)$, then since $G$ is ZF-dense, there exists a minimum zero forcing set $B_G$ of $G$ with $v \in B_G$.  Likewise, if $v \in V(H_i)$ for some $i$, then since $H$ is ZF-dense there exists $B_H$ a minimum zero forcing set of $H$ with $v \in B_H^i$, where for each $i \in [q]$, $B_H^i$ is the copy of $B_H$ in $H_i$.  By Lemma \ref{lem:coronasets} and Proposition \ref{coronaval}, $B=B_G \cup (\bigcup_{i=1}^{q}B_H^i)$ is a minimum zero forcing set of $G \circ H$ containing $v$.

Next suppose that $H$ is TF-dense and $Z_t(H)=Z(H)$.  If $v \in V(G)$, then since $G$ is ZF-dense, there exists a minimum zero forcing set $B_G$ of $G$ with $v \in B_G$.  Likewise, if $v \in V(H_i)$ for some $i$, then since $H$ is TF-dense there exists $B_H$ a minimum total forcing set of $H$ with $v \in B_H^i$.  By Lemmas \ref{lem:coronasets} and \ref{totalcoronaval}, $B=B_G \cup (\bigcup_{i=1}^{q}B_H^i)$ is a minimum total forcing set of $G \circ H$ containing $v$.

Finally suppose that $H$ is ZF-dense and $Z_t(H)>Z(H)$.  If $v \in V(H_i)$ for some $i$, then since $H$ is ZF-dense there exists $B_H$ a minimum zero forcing set of $H$ with $v \in B_H^i$.  By Lemmas \ref{lem:coronasets} and \ref{totalcoronaval}, $B=V(G) \cup (\bigcup_{i=1}^{q}B_H^i)$ is a minimum total forcing set of $G \circ H$ containing $v$.
\end{proof}

\begin{thm}
\label{thmcorcon}
Let $G$ and $H$ be connected graphs, with $H$ nontrivial.  If $H$ is CF-dense, then $G \circ H$ is CF-dense.
\end{thm}

\begin{proof}
Begin by enumerating the vertices of $G$, $\{u_i\}_{i=1}^q$.  Next, for each $i \in [q]$, let $H_i'=(G \circ H)[V(H_i) \cup \{u_i\}]$.  Now let $v \in V(G \circ H)$ be arbitrary.  If $v \in V(H_i)$ for some $i$, then since $H$ is CF-dense there exists $B_H$ a minimum connected forcing set of $H$ with $v \in B_H^i$, where for each $i \in [q]$, $B_H^i$ is the copy of $B_H$ in $H_i$.  By Lemmas \ref{lem:coronasets} and \ref{connectcoronaval}, $B=V(G) \cup (\bigcup_{i=1}^{q}B_H^i)$ is a minimum connected forcing set of $G \circ H$ containing $v$.
\end{proof}

Theorems \ref{thmcortfzf} and \ref{thmcorcon} imply the following corollary.

\begin{cor}
Let $G$ and $H$ be connected graphs, with $H$ nontrivial.  If $G$ is ZF-dense and $H$ is ZTCF-dense, then $G \circ H$ is ZTCF-dense. 
\end{cor}

Therefore, the corona of any pair of ZTCF-dense graphs identified in Section \ref{sec:basics} is also ZTCF-dense.

\section{Conclusion}
\label{sec:conclusion}
In this paper, we introduced and studied the concept of connected forcing density, extending prior work on zero forcing and total forcing to incorporate connectivity constraints. In Section \ref{sec:uniqueness}, we gave necessary conditions for the uniqueness of connected forcing sets, and characterized the number of distinct connected forcing sets of extremal sizes. In Section \ref{sec:basics}, we identified a rich collection of ZTCF-dense graphs, including a family in which $Z(G) < Z_t(G) < Z_c(G)$. In Section \ref{sec:trees-and-unicyclic}, we characterized CF-dense trees. We also derived a closed form expression for counting the number of connected forcing sets of any size in trees. In Section \ref{sec:cartesian-alt}, we established structural results for constructing new ZTCF-dense graphs via Cartesian products, joins, and coronas.  We also introduced results concerning the total and connected forcing numbers of Cartesian products, joins, and coronas of graphs.

Our results underscore the combinatorial richness of connected forcing and its interplay with structural graph properties. Future directions include classifying CF-dense graphs within other graph families (such as chordal graphs, unicyclic graphs, or planar graphs), exploring probabilistic models of CF-density in random graphs, and further investigating algorithmic and complexity aspects of identifying minimum connected forcing sets and CF-dense graphs.


\medskip
\begin{thebibliography}{99}

\bibitem{target1}
E. Ackerman, O. Ben-Zwi and G. Wolfovitz,
Combinatorial model and bounds for target set selection,
\emph{Theoret. Comput. Sci.}, {\bf 411}(44-46)
(2010), 4017--4022.

\bibitem{AIM-Workshop} AIM Minimum Rank-Special Graphs Work Group, Zero forcing sets and the minimum rank of graphs, \textit{Linear Algebra Appl.}, \textbf{428}(7) (2008), 1628--1648.

\bibitem{k-forcing} D. Amos, Y. Caro, R. Davila and R. Pepper, Upper bounds on the k-forcing number of a graph, \emph{Discrete Appl. Math.}, \textbf{181} (2015), 1--10.

\bibitem{zf_tw}
F. Barioli,  W. Barrett, S.M. Fallat, T. Hall, L. Hogben, B. Shader, P. van den Driessche and H. van der Holst,
Parameters related to tree-width, zero forcing, and maximum nullity of a graph,
\emph{J. Graph Theory}, \textbf{72}(2) (2013),
146--177.

\bibitem{target3}
O. Ben-Zwi, D. Hermelin, D. Lokshtanov and I. Newman,
Treewidth governs the complexity of target set selection, 
\emph{Discrete Optim.}, \textbf{8}(1) (2011),
87--96.

\bibitem{upper-total-dom} B. Brešar, M. G. Cornet, T. Dravec and M. A. Henning, Bounds on zero forcing using (upper) total domination and minimum degree, \emph{Bull. Malays. Math. Sci. Soc.}, \textbf{47} (2024) no. 143.

\bibitem{BrDa} B. Brimkov and R. Davila, Characterizations of the Connected Forcing Number of a Graph, \emph{arXiv:1604.00740}, (2016).

\bibitem{vc} B. Brimkov, R. Davila, H. Schuerger and M. Young.  On a conjecture of TxGraffiti: relating zero forcing and vertex covers in graphs. \emph{Discrete Appl. Math.}, {\bf 359} (2024), 290--302

\bibitem{extremalcf}
B. Brimkov, C. C. Fast and I. V. Hicks, Graphs with Extremal Connected Forcing Numbers, \emph{arXiv:1604.00740}, (2017).

\bibitem{cf-complexity} B. Brimkov and I. V. Hicks, Complexity and computation of connected zero forcing, \emph{Discrete Appl. Math.}, \textbf{229} (2017), 31--45.

\bibitem{BruneiHeath}
D.J. Brueni and L.S. Heath,
The {PMU} placement problem, 
\emph{SIAM J. Discrete Math.}, \textbf{19}(3) (2005),
744--761.

\bibitem{quantum1}
D. Burgarth and V. Giovannetti,
Full control by locally induced relaxation,
\emph{Phys. Rev. Lett.}, \textbf{99}(10) (2007), 
100501.

\bibitem{logic1}
D. Burgarth, V. Giovannetti,  L. Hogben, S. Severini and  M. Young,
Logic circuits from zero forcing,
\emph{Nat. Comput.}, \textbf{14} (2015), 485--490.

\bibitem{cartfrac} T. Cameron, L. Hogben, F. Kenter, S.A. Mojalall and H. Schuerger, Forts, (fractional) zero forcing, and Cartesian products of graphs, \emph{arxiv:2310.17904}, (2023).

\bibitem{NP-Complete} C. Chekuri and N. Korula, A graph reduction step preserving element-connectivity and applications, \emph{Automata, Languages, and Programming}, (2009) 254--265. 

\bibitem{target2}
C.Y. Chiang, L.H. Huang, B.J. Li, J. Wu and H.G. Yeh,
Some results on the target set selection problem, 
\emph{J. Comb. Optim.}, \textbf{25}(4) (2013),
702--715.

\bibitem{Da-15} R. Davila, Bounding the forcing number of a graph. {\em Rice University, Masters Thesis}, 2015.

\bibitem{Da-19} R. Davila, Total and zero forcing in graphs. \emph{Ph.D Thesis, University of Johannesburg}, 2019.

\bibitem{DaHe19c} R. Davila and M. A. Henning, Matching, path covers, and total forcing sets, \emph{Quaest. Math.}, \textbf{43}(1) (2019), 131--147.

\bibitem{DaHe19a} R. Davila and M. A. Henning, On the total forcing number of a graph, \textit{Discrete Appl. Math.}, \textbf{257} (2019), 115--127.

\bibitem{DaHe-20b} R. Davila and M. A. Henning, Relating zero forcing and domination in cubic graphs, \textit{J. Comb. Optim.}, \textbf{41} (2021), 553--577.

\bibitem{DaHe18b} R. Davila and M. A. Henning, Total forcing and zero forcing in claw-free cubic graphs, \textit{Graphs Combin.}, \textbf{34} (2018), 1371--1384.

\bibitem{DaHe19b} R. Davila and M. A. Henning, Total forcing versus total domination in cubic graphs, \textit{Appl. Math. Comput.}, \textbf{354} (2019), 385--395.

\bibitem{DaHe-20a} R. Davila and M. A. Henning, Zero forcing in claw-free cubic graphs, \textit{Bull. Malays. Math. Sci. Soc.}, \textbf{43} (2020), 673--688.

\bibitem{Da-connected-forcing} R. Davila, M. A. Henning, C. Magnant and R. Pepper, Bounds on the connected forcing number of a graph, \emph{Graphs Combin.}, \textbf{34}(6) (2018), 1159--1174.

\bibitem{DaHePe2023a} R. Davila, M.A. Henning and R. Pepper, Zero and total forcing dense graphs, \textit{Discuss. Math. Graph Theory}, \textbf{43} (2023) 619--634.

\bibitem{zfclaw} R. Davila, H. Schuerger and B. Small, A characterization of claw-free graphs using zero forcing invariants, \emph{arxiv:2412.0343}, (2024).

\bibitem{powerdom3}
T. Haynes,  S. Hedetniemi, S. Hedetniemi and  M. Henning,
Domination in graphs applied to electric power networks,
\emph{SIAM J. Discrete Math.}, \textbf{15}(4) (2002), 
519--529.

\bibitem{HaHeHe2024} T. Haynes, S. Hedetniemi and M. A. Henning, \emph{Domination in Graphs: Core Concepts (Springer Monographs in Mathematics)}. (2023).

\bibitem{HeLo12} M. A. Henning and C. L\"{o}wenstein, Locating-total domination in claw-free cubic graphs, \textit{Discrete Math.}, \textbf{312} (2012), 3107--3116.

\bibitem{HoLiSh-zero-forcing-book} L. Hogben, J. C.-H. Lin and B. L. Shader, \emph{Inverse Problems and Zero Forcing for Graphs (AMS Mathematical Surveys and Monographs, 270)}. (2022). 

\bibitem{ZGcircKs} I. Javaid, I.~Irshad, M.~Batool and Z.~Raza, On the zero forcing number of corona and lexicographic product of graphs, \emph{arxiv:1607.04071}, (2016).

\bibitem{new-connected-zero-forcing} M. Khosravi, S. Rashidi and A. Sheikhhosseini, Connected zero forcing sets and connected propagation time of graphs, \emph{Trans. Comb.}, \textbf{9}(2) (2020), 77--88.

\bibitem{Pe12} T. Peters, Positive semidefinite maximum nullity and zero forcing number, \textit{Electron. J. Linear Algebra}, \textbf{23} (2012), 815--830.

\bibitem{indep} H. Schuerger, N. Warnberg and M. Young, Zero forcing and vertex independence number on cubic and subcubic graphs, \emph{arxiv:2410.21724}, (2024).

\bibitem{join} F. A. Taklimi, Zero forcing sets for graphs. \emph{University of Regina, Ph.D Thesis}, 2013. 

\bibitem{zf_np}
M. Trefois and J.C. Delvenne,
Zero forcing number, constrained matchings and strong structural controllability,
\emph{Linear Algebra Appl.},
\textbf{484} (2015), 199--218.

\bibitem{fast_mixed_search}
B. Yang,
Fast-mixed searching and related problems on graphs,
\emph{Theoret. Comput. Sci.}, \textbf{507} (2013), 100--113.

\end{thebibliography}
\end{document}